\documentclass[11pt]{article}

\usepackage{enumerate,xspace}
\usepackage{amsmath,amssymb,wasysym}
\usepackage[all]{xy}
\usepackage{proof}
\usepackage[svgnames]{xcolor}
\usepackage{tikz}

\usepackage{stmaryrd} %need for special interpretation bracket
\usepackage{mathtools}
\usepackage{latexsym}
\usepackage{dsfont}
\usepackage{multicol}
\usepackage{cmll}
\usepackage{multirow}
\usepackage{longtable}

\usepackage{lscape}
\usepackage{array}

\usepackage{tikz}
\usetikzlibrary{tikzmark}
\usetikzlibrary{shapes.geometric, arrows}
\usepackage{tikz-cd}
\usetikzlibrary{cd}
\usepackage{framed,color}
\usepackage{smartdiagram}
\usesmartdiagramlibrary{additions}
\usetikzlibrary{positioning,arrows}
\usetikzlibrary{shapes}
\usetikzlibrary{shadows.blur}
\usetikzlibrary{shapes.symbols}
\usetikzlibrary{arrows.meta,calc}
\usetikzlibrary{shapes,snakes}
\pgfdeclarelayer{edgelayer}
\pgfdeclarelayer{nodelayer}
\pgfsetlayers{edgelayer,nodelayer,main}

\usepackage{hyperref} %hyperlink package
\hypersetup{
    colorlinks,
    citecolor=red,
    filecolor=red,
    linkcolor=blue,
    urlcolor=red
}

\newtheorem{observation}{Remark}[section]
\newtheorem{theorem}[observation]{Theorem}
\newtheorem{definition}[observation]{Definition}
\newtheorem{example}[observation]{Example}

\newtheorem{proposition}[observation]{Proposition}

\makeatletter

\newdimen\w@dth

\def\setw@dth#1#2{\setbox\z@\hbox{\scriptsize $#1$}\w@dth=\wd\z@
\setbox\@ne\hbox{\scriptsize $#2$}\ifnum\w@dth<\wd\@ne \w@dth=\wd\@ne \fi
\advance\w@dth by 1.2em}

\def\t@^#1_#2{\allowbreak\def\n@one{#1}\def\n@two{#2}\mathrel
{\setw@dth{#1}{#2}
\mathop{\hbox to \w@dth{\rightarrowfill}}\limits
\ifx\n@one\empty\else ^{\box\z@}\fi
\ifx\n@two\empty\else _{\box\@ne}\fi}}
\def\t@@^#1{\@ifnextchar_ {\t@^{#1}}{\t@^{#1}_{}}}

\def\t@left^#1_#2{\def\n@one{#1}\def\n@two{#2}\mathrel{\setw@dth{#1}{#2}
\mathop{\hbox to \w@dth{\leftarrowfill}}\limits
\ifx\n@one\empty\else ^{\box\z@}\fi
\ifx\n@two\empty\else _{\box\@ne}\fi}}
\def\t@@left^#1{\@ifnextchar_ {\t@left^{#1}}{\t@left^{#1}_{}}}

\def\two@^#1_#2{\def\n@one{#1}\def\n@two{#2}\mathrel{\setw@dth{#1}{#2}
\mathop{\vcenter{\hbox to \w@dth{\rightarrowfill}\kern-1.7ex
                 \hbox to \w@dth{\rightarrowfill}}%
       }\limits
\ifx\n@one\empty\else ^{\box\z@}\fi
\ifx\n@two\empty\else _{\box\@ne}\fi}}
\def\tw@@^#1{\@ifnextchar_ {\two@^{#1}}{\two@^{#1}_{}}}

\def\tofr@^#1_#2{\def\n@one{#1}\def\n@two{#2}\mathrel{\setw@dth{#1}{#2}
\mathop{\vcenter{\hbox to \w@dth{\rightarrowfill}\kern-1.7ex
                 \hbox to \w@dth{\leftarrowfill}}%
       }\limits
\ifx\n@one\empty\else ^{\box\z@}\fi
\ifx\n@two\empty\else _{\box\@ne}\fi}}
\def\t@fr@^#1{\@ifnextchar_ {\tofr@^{#1}}{\tofr@^{#1}_{}}}

\newdimen\W@dth
\def\setW@dth#1#2{\setbox\z@\hbox{$#1$}\W@dth=\wd\z@
\setbox\@ne\hbox{$#2$}\ifnum\W@dth<\wd\@ne \W@dth=\wd\@ne \fi
\advance\W@dth by 1.2em}

\def\T@^#1_#2{\allowbreak\def\N@one{#1}\def\N@two{#2}\mathrel
{\setW@dth{#1}{#2}
\mathop{\hbox to \W@dth{\rightarrowfill}}\limits
\ifx\N@one\empty\else ^{\box\z@}\fi
\ifx\N@two\empty\else _{\box\@ne}\fi}}
\def\T@@^#1{\@ifnextchar_ {\T@^{#1}}{\T@^{#1}_{}}}

\def\T@left^#1_#2{\def\N@one{#1}\def\N@two{#2}\mathrel{\setW@dth{#1}{#2}
\mathop{\hbox to \W@dth{\leftarrowfill}}\limits
\ifx\N@one\empty\else ^{\box\z@}\fi
\ifx\N@two\empty\else _{\box\@ne}\fi}}
\def\T@@left^#1{\@ifnextchar_ {\T@left^{#1}}{\T@left^{#1}_{}}}

\def\Tofr@^#1_#2{\def\N@one{#1}\def\N@two{#2}\mathrel{\setW@dth{#1}{#2}
\mathop{\vcenter{\hbox to \W@dth{\rightarrowfill}\kern-1.7ex
                 \hbox to \W@dth{\leftarrowfill}}%
       }\limits
\ifx\N@one\empty\else ^{\box\z@}\fi
\ifx\N@two\empty\else _{\box\@ne}\fi}}
\def\T@fr@^#1{\@ifnextchar_ {\Tofr@^{#1}}{\Tofr@^{#1}_{}}}

\def\Two@^#1_#2{\def\N@one{#1}\def\N@two{#2}\mathrel{\setW@dth{#1}{#2}
\mathop{\vcenter{\hbox to \W@dth{\rightarrowfill}\kern-1.7ex
                 \hbox to \W@dth{\rightarrowfill}}%
       }\limits
\ifx\N@one\empty\else ^{\box\z@}\fi
\ifx\N@two\empty\else _{\box\@ne}\fi}}
\def\Tw@@^#1{\@ifnextchar_ {\Two@^{#1}}{\Two@^{#1}_{}}}

\def\to{\@ifnextchar^ {\t@@}{\t@@^{}}}
\def\from{\@ifnextchar^ {\t@@left}{\t@@left^{}}}
\def\tofro{\@ifnextchar^ {\t@fr@}{\t@fr@^{}}}
\def\To{\@ifnextchar^ {\T@@}{\T@@^{}}}
\def\From{\@ifnextchar^ {\T@@left}{\T@@left^{}}}
\def\Two{\@ifnextchar^ {\Tw@@}{\Tw@@^{}}}
\def\Tofro{\@ifnextchar^ {\T@fr@}{\T@fr@^{}}}

\makeatother

\title{Tangent Categories from the Coalgebras of Differential Categories}
\author{Robin Cockett \footnote{Partially supported by NSERC (Canada)} , Jean-Simon Pacaud Lemay \footnote{Thanks to Kellogg College, the Clarendon Fund, and the Oxford Google-DeepMind Graduate Scholarship for financial support.} , and Rory B. B. Lucyshyn-Wright \footnote{Supported by the Natural Sciences and Engineering Research Council of Canada (NSERC).}}

\begin{document}
\allowdisplaybreaks

\maketitle
%\section{}
%\subsection{}

%\tableofcontents
%\newpage

%%%%%%%%%%%%%%%%%%%%%%%%%%%%%%%%%%%%%%%%%%%%%%%%%%%%%%%%%%%%%%%%%%%%%

\begin{abstract}
Following the pattern from linear logic, the coKleisli category of a differential category is a Cartesian differential category.   What then is the coEilenberg-Moore category of 
a differential category?  The answer is a tangent category!  A key example arises from the opposite of  the category of Abelian groups with the free exponential modality. 
The coEilenberg-Moore category, in this case, is the opposite of the category of commutative rings. That the latter is a tangent category captures a fundamental aspect 
of both algebraic geometry and Synthetic Differential Geometry. The general result applies when there are no negatives and thus encompasses examples arising from 
combinatorics and computer science. This is an extended version of a conference paper for CSL2020. 
\end{abstract}

\section{Introduction}

\begin{figure*}[!t]
\centering
\begin{tikzpicture}[->,>=stealth',shorten >=1pt,node distance=0.5cm,auto,
every node/.style = {shape=rectangle, rounded corners,% is not necessary, default node's shape is rectangle
                     align=center}]
\node (pq1) [draw=black, thick] {\footnotesize Differential Categories\\ \footnotesize Blute, Cockett, Seely \cite{blute2006differential}};
\node (pq2) [draw=black, thick, right=of pq1] {\footnotesize Cartesian Differential Categories\\ \footnotesize Blute, Cockett, Seely \cite{blute2009cartesian}};
\node (pq3) [draw=black, thick, above right=of pq2] {\footnotesize Restriction Differential Categories\\ {\footnotesize Cockett, Cruttwell, Gallagher \cite{cockett2011differential}}};
\node (pq4) [draw=black, thick, below right=of pq2, below=4cm of pq3] {\footnotesize Tangent Categories\\ {\footnotesize Rosicky \cite{rosicky1984abstract}} \\ {\footnotesize Cockett, Cruttwell \cite{cockett2014differential}}};
\path[->,thick] (pq1) edge[bend left=45] node[midway, above, sloped] {\footnotesize coKleisli}  (pq2)
(pq2) edge[bend left=40] node[midway, below, sloped] {\footnotesize $\otimes$-Representation}  (pq1)
                     (pq3) edge[bend left=20] node[midway, above, sloped] {\footnotesize Manifold 
                        Completion} (pq4)
                        (pq2.east)  edge[bend left=25] node [midway, above, sloped] (TextNode) {$\subset$} (pq4) 
                          (pq2.east)  edge[bend right=25] node [midway, below, sloped] {$\subset$} (pq3)
                           (pq3)  edge[bend right=20] node [midway, above, sloped] {\footnotesize Total
                           Maps} (pq2) 
                             (pq4.west)  edge[bend left=25] node [midway, below, sloped] {\footnotesize Differential 
                           Objects} (pq2) 
                                 (pq1.south) edge[bend right=35] node [midway, below, sloped] {\footnotesize coEilenberg-Moore \\
                                \textbf{\footnotesize Story of this paper}}  (pq4);
        \end{tikzpicture}
        \label{worldofdiffcats}
        \caption{The world of differential categories and how it's all connected.}
\end{figure*}
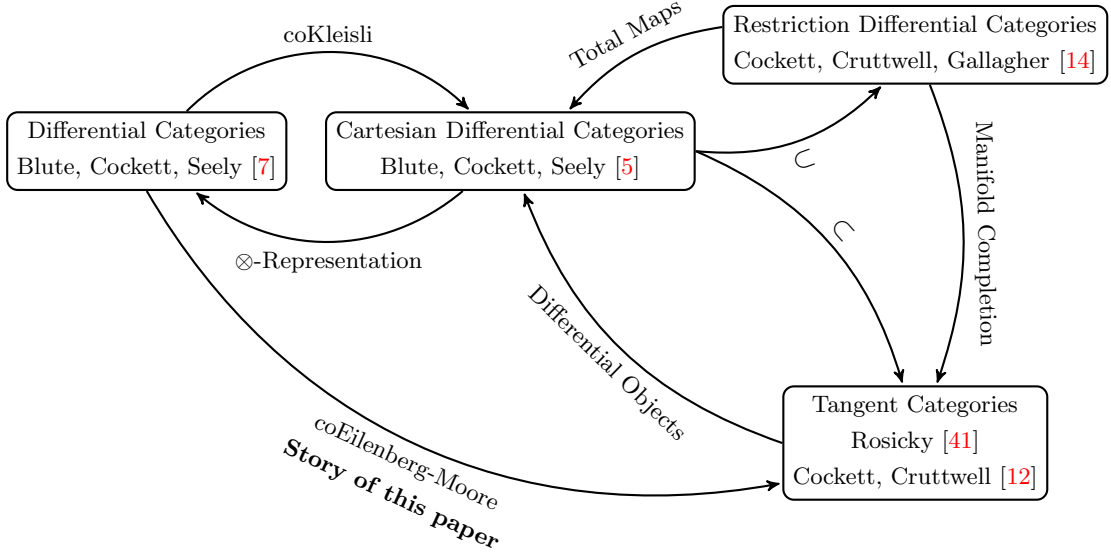

It is well-established, following the pattern from linear logic \cite{girard1987linear}, that the coKleisli category of a (tensor) differential category is a Cartesian differential category \cite{blute2009cartesian}.   
What then is the coEilenberg-Moore category of a (tensor) differential category? The answer, which is the subject of this paper, is that, under mild limit assumptions, it is a tangent category \cite{cockett2014differential}. A tangent category is a category equipped with an endofunctor and several natural transformations that formalize the basic properties of the tangent bundle functor on the category of smooth manifolds. The study of the tangent category structure of coEilenberg-Moore categories of such differential categories was initiated by the third author in talks at the Category Theory 2014 conference \cite{roryCT2014} and the Foundational Methods in Computer Science conference in 2014 \cite{roryFMCS}.  

The use of differentiation in programming, particularly for applications in machine learning, has renewed computer scientists' interest in the abstract semantics of differentiation.   Tensor differential categories provide perhaps the simplest abstract description of differentiation.  Tangent categories, on the other hand, are at quite the other end of the spectrum, providing an abstract semantics for differential geometry.  That the two are directly linked by the coEilenberg-Moore construction (which is purely algebraic) witnesses that there is a surprisingly direct relationship between differential programming and differential geometry which might usefully be exploited.

In Figure 1 the relationships between the various differential categories are illustrated. The investigation of differential structure of this kind was initiated by Erhrard \cite{ehrhard2002kothe} and formulated as a categorical axiomatization in \cite{blute2006differential}. Classical smooth functions arose indirectly as the coKleisli maps of these differential categories: thus, the next step was to directly 
axiomatize this classical intuition. This was accomplished in \cite{blute2009cartesian}, where Cartesian differential categories were introduced. By considering the representability of a tensor product in Cartesian differential 
categories it was then possible to extract a (tensor) differential category from a Cartesian differential category \cite{blute2015cartesian}. Classical analysis considers maps that are not defined everywhere 
and, thus, the theory of Cartesian differential categories with partiality was developed \cite{cockett2011differential}.  It was a natural step from there to consider differentiable manifolds, and this created a desire to develop a basic 
axiomatization for differential geometry: this led to the development of abstract differential geometry based on tangent categories, which had been introduced by Rosick{\'y} in \cite{rosicky1984abstract} and were later generalized and studied further by Cockett and Cruttwell in \cite{cockett2014differential}.  An important alternative approach to tangent categories was introduced by Leung \cite{leung2017classifying} and was further developed by Garner in \cite{garner2018embedding} to provide a view of tangent categories as categories enriched in Weil spaces.  Cartesian differential categories are always examples of tangent categories \cite[Proposition 4.7]{cockett2014differential}. Conversely the differential bundles of a tangent category over a fixed base, under mild limit assumptions,  form a Cartesian differential category, showing that a tangent category is {\em locally\/} a Cartesian differential category \cite[Theorem 5.14]{cockett2016differential}. These observations tightly linked Cartesian differential and tangent categories; in fact this relationship is captured by an adjunction \cite[Theorem 4.12]{cockett2014differential}. The current paper provides an important direct link between (tensor) differential categories and tangent categories. 

An important example of a (tensor) differential category is the opposite of the category of Abelian groups with the free exponential modality \cite{blute2006differential} where the differential structure is based on
differentiating polynomials. The coEilenberg-Moore category, in this case, is the opposite of the category of commutative rings.  The fact that this is a tangent category captures a fundamental aspect of both algebraic geometry \cite{hartshorne2013algebraic} and Synthetic Differential Geometry \cite{kock2006synthetic}.  That the coEilenberg-Moore category of a differential category is a tangent category in much more generality allows further significant examples. Not only can one dispense with the necessity 
of assuming {\em negatives\/}, but also with the necessity of having a {\em monoidal\/} coalgebra modality \cite{blute2015cartesian} or, equivalently, the Seely isomorphisms, $\oc(A \times B) \cong \oc(A) \otimes \oc(B)$ \cite{seely1987linear} (which the third author required in \cite{roryFMCS}).  Dispensing with the assumption of negatives allows one to generalize the example of commutative rings to commutative semirings \cite{golan2013semirings} and to consider examples from combinatorics and computer science. Dispensing with the assumption of a monoidal coalgebra modality/the Seely isomorphisms allows consideration of such examples as $C^\infty$-rings \cite{kock2006synthetic,moerdijk2013models} or Rota-Baxter algebras \cite{guo2012introduction}. When a coalgebra modality is monoidal, it will give rise, when sufficient limits are present, to a \emph{representable tangent category}. This means that the tangent functor is of the form $\_^D$ for an infinitesimal object $D$  and so has left adjoint $\_ \times D$. Two examples of such differential categories include the opposite category of vector spaces with the free commutative algebra modality (one of the original examples of a differential category in \cite{blute2006differential}), as well as the the category of vector spaces with the cofree cocommutative coalgebra modality (as studied by Clift and Murfet in \cite{clift2017cofree}). It is interesting to note that in both cases, the infinitesimal object is the ring of dual numbers over the field.  On the other hand, as the free $C^\infty$-ring modality is not monoidal, this provides an example of a non-representable tangent category that, nonetheless, has a tangent functor that is a right adjoint.

\noindent \textbf{Conventions:} This paper assumes a knowledge of basic category theory and of symmetric monoidal categories. We refer the reader to \cite{mac2013categories} if further details are needed on these topics. In this paper we shall use diagrammatic order for composition: explicitly, this means that the composite map $fg$ is the map that employs $f$ first and then $g$. Furthermore, to simplify working in symmetric monoidal categories, we will allow ourselves to work in symmetric \emph{strict} monoidal categories, and therefore we suppress the associator and unitor isomorphisms. Symmetric monoidal categories will be denoted by $(\mathbb{X}, \otimes, K, \tau)$ where $\mathbb{X}$ is the underlying category, $\otimes$ is the tensor product, $K$ is the monoidal unit, and $\tau_{A,B}: A \otimes B \to B \otimes A$ is the symmetry isomorphism.

\section{Tangent Categories}\label{tansec}

Tangent categories were introduced by Rosick{\'y} in \cite{rosicky1984abstract}, then later generalized and studied further by Cockett and Cruttwell in \cite{cockett2014differential}. This generalization, which replaced Abelian groups with commutative monoids, opened the door to examples of tangent categories from combinatorics and computer science where one does not expect to have negatives. The axioms of a tangent category abstract the essential properties of tangent bundles over smooth manifolds \cite{lee2009manifolds}. In this section we provide a brief overview of tangent categories, and refer the reader to \cite{cockett2014differential, garner2018embedding} for a more in-depth introduction. 

\begin{definition} Let $\mathbb{X}$ be a category. A \textbf{tangent structure} \cite{cockett2014differential} $\mathbb{T}$ on $\mathbb{X}$ is a sextuple $\mathbb{T} := (\mathsf{T}, p, \sigma, z, \ell, c)$ consisting of: 
\begin{itemize}
\item An endofunctor $\mathsf{T}: \mathbb{X} \to \mathbb{X}$;
\item A natural transformation $p_M: \mathsf{T}(M) \to M$, known as the \textbf{projection}, such that for each $M$ and each $n \in \mathbb{N}$, there is an $n$-th \textit{fibre power}\footnote{I.e, a fibered product \cite{mac2013categories} of $n$ instances of $p_M$} $\mathsf{T}_n(M)$ of $p_M$ (with projections $\rho_i: \mathsf{T}_n(M) \to \mathsf{T}(M)$), and this fibre power is preserved by $\mathsf{T}^m$ for each $m \in \mathbb{N}$. For each $n \in \mathbb{N}$, this induces a functor $\mathsf{T}_n: \mathbb{X} \to \mathbb{X}$ where by convention $\mathsf{T}_0 = 1_\mathbb{X}$ and $\mathsf{T}_1 = \mathsf{T}$. 
%$$ \xymatrixcolsep{4pc}\xymatrixrowsep{1pc}\xymatrix{ & \mathsf{T}_n(M) \ar[dl]_-{\rho_0}  \ar[d]_-{\rho_j} \ar[dr]^-{\rho_{n-1}}  \\ 
 % \mathsf{T}(M) \ar[dr]_-{p_{M}} &   \mathsf{T}(M)  \ar@{}[r]^-{\hdots} \ar@{}[l]_-{\hdots} \ar[d]^-{p_{M}} &   \mathsf{T}(M) \ar[dl]^-{p_{M}} \\
 % &  M
  %} $$
  \item Natural transformations ${\sigma_M: \mathsf{T}_2(M) \to \mathsf{T}(M)}$, known as the \textbf{sum operation on tangent vectors}, ${z_M: M \to \mathsf{T}(M)}$, known as the \textbf{zero vector field}, ${\ell_M: \mathsf{T}(M) \to \mathsf{T}^2(M)}$, known as the \textbf{vertical lift}, and $c_M: \mathsf{T}^2(M) \to \mathsf{T}^2(M)$, known as the \textbf{canonical flip};
\end{itemize}
and such that $p$, $\sigma$, $z$, $\ell$, and $c$ satisfy the various equational axioms found in \cite{cockett2014differential, garner2018embedding} and that for each $M$, the following diagram is an equalizer diagram \cite{mac2013categories}, known as the \textbf{universality of the vertical lift}: 
$$ \xymatrixcolsep{5pc}\xymatrix{ \mathsf{T}_2 (M) \ar[rr]^-{\langle \rho_0 z_{\mathsf{T}(M)}, \rho_1 \ell_M \rangle\mathsf{T}(\sigma_M)} && \mathsf{T}^2(M) \ar@<+.5ex>[r]^-{\mathsf{T}(p_M)} \ar@<-.5ex>[r]_-{p_{\mathsf{T}(M)} p_M z_M}
 & \mathsf{T}(M)} $$
where $\langle -, - \rangle$ is the pairing operation induced by the universal property of the pullback. 
\end{definition}

\begin{definition} A \textbf{tangent category} \cite{cockett2014differential} is a pair $(\mathbb{X}, \mathbb{T})$ consisting of a category $\mathbb{X}$ and a tangent structure $\mathbb{T}$ on $\mathbb{X}$. The fibre powers of $p$, together with the equalizer appearing in the axiom of universality of the vertical lift, are collectively referred to as the \textbf{tangent limits} \cite{garner2018embedding} of a tangent category. 
\end{definition}

We refer the reader to \cite{cockett2014differential} where the axioms of a tangent category are expressed in commutative diagrams. In \cite{leung2017classifying}, Leung defined an alternative axiomatization of a tangent category using Weil algebras; this was exploited by Garner in \cite{garner2018embedding} to provide a description of tangent categories as categories enriched in Weil spaces. 
\begin{example} \normalfont\label{tanex} Here are some well-known examples of tangent categories. Other examples of tangent categories can be found in \cite{cockett2014differential,garner2018embedding, cockett2016differential}. 
\begin{enumerate}
\item The canonical example of a tangent category is the category of finite-dimensional smooth manifolds, where for a manifold $M$, $\mathsf{T}(M)$ is the standard tangent bundle over $M$. 
\item \label{Cartex} Every Cartesian differential category \cite{blute2009cartesian} is a tangent category \cite[Proposition 4.7]{cockett2014differential}. In particular this implies that every categorical model of the differential $\lambda$-calculus \cite{bucciarelli2010categorical, manzonetto2012categorical} is a tangent category.
\item \label{Crigex} Let $k$ be a field, and let $\mathsf{CALG}_k$ be the category of commutative $k$-algebras. Then $\mathsf{CALG}_k$ is a tangent category where for a commutative $k$-algebra $A$, $\mathsf{T}(A) := A[\epsilon]$ is the ring of dual numbers over $A$, $A[\epsilon] = \lbrace a + b \epsilon \vert~ a,b \in A\rbrace$ with $\epsilon^2 = 0$. The projection is defined as ${p_A(a + b \epsilon) = a}$, and so ${\mathsf{T}_2(A) := A [\epsilon, \epsilon^\prime] =  \lbrace a + b \epsilon + c \epsilon^\prime \vert~ a,b,c \in A\rbrace}$ with $\epsilon^2 =  {\epsilon^\prime}^2 = \epsilon\epsilon^\prime = 0$. 
On the other hand, ${\mathsf{T}^2(A) := \lbrace a + b \epsilon_1 + c \epsilon_2 + d \epsilon_1\epsilon_2 \vert~ a,b,c,d \in A\rbrace}$ with $\epsilon_1^2 = \epsilon_2^2 = 0$. The remaining tangent structure is defined as follows: $\sigma_A(a + b\epsilon + c\epsilon^\prime) =  a + (b+c)\epsilon$, $z_A (a) = a$, ${\ell_A(a + b \epsilon) = a + b \epsilon_1\epsilon_2}$, and $c_A(a + b \epsilon_1 + c \epsilon_2 + d \epsilon_1\epsilon_2) = a + c \epsilon_1 + b \epsilon_2 + d \epsilon_1\epsilon_2$. We will generalize this example in the context of codifferential categories in Section \ref{EMCodiffsec}. In particular, this example generalizes to the category of commutative algebras over any commutative unital semiring.
\end{enumerate}
 \end{example}

Of particular importance to this paper is when the tangent functor has a left adjoint, which induces a tangent structure on the opposite category of the tangent category. %We denote adjunctions as quadruples  $(\mathsf{L}, \mathsf{R}, \alpha, \beta)$ where $\mathsf{L}$ is a left adjoint to $\mathsf{R}$, $\alpha_X: X \to \mathsf{R}\mathsf{L}(X)$ is unit of the adjunction, and $\beta_Y: \mathsf{L}\mathsf{R}(Y) \to Y$ is the counit of the adjunction. 

\begin{theorem}\label{tanadj} \cite[Proposition 5.17]{cockett2014differential} Let $(\mathbb{X}, \mathbb{T})$ be a tangent category such that for each $n \in \mathbb{N}$, $\mathsf{T}_n: \mathbb{X} \to \mathbb{X}$ has a left adjoint $\mathsf{L}_n: \mathbb{X} \to \mathbb{X}$. Then $\mathbb{X}^{op}$ has a tangent structure with tangent functor $\mathsf{L} = \mathsf{L_1}$.  
\end{theorem} 

See Example \ref{reptaneex}.\ref{Crigop} for an application of this theorem. In Section \ref{CoEMdiffsec} we will use Theorem \ref{tanadj} to obtain a tangent structure on the coEilenberg-Moore category of a differential category. 

We now turn our attention to \emph{representable tangent categories}, which briefly are tangent categories whose tangent functor is representable. Representable tangent categories are a very important kind of tangent category as they are closely related to synthetic differential geometry \cite{kock2006synthetic}. First recall that in a category $\mathbb{X}$ with binary products $\times$, an object $D$ is an \textbf{exponent object} if the functor $- \times D: \mathbb{X} \to \mathbb{X}$ has a right adjoint $(-)^D: \mathbb{X} \to \mathbb{X}$. A functor $\mathsf{F}: \mathbb{X} \to \mathbb{X}$ is \textbf{representable} if $\mathsf{F}(-) \cong (-)^D$ for some exponent object $D$, and $D$ is said to represent the functor $\mathsf{F}$. 

\begin{definition} A \textbf{representable tangent category} \cite{cockett2014differential} is a tangent category $(\mathbb{X}, \mathbb{T})$ such that $\mathbb{X}$ has finite products and for each $n \in \mathbb{N}$, $\mathsf{T}_n$ is a representable functor, that is, $\mathsf{T}_n(-) \cong (-)^{D_n}$ for some exponent object $D_n$. In the case of $n=1$, the object $D_1$ (which we denote simply as $D$) representing the tangent functor, $\mathsf{T}(-) \cong (-)^D$, is called the \textbf{infinitesimal object} \cite{cockett2014differential} of the representable tangent category $(\mathbb{X}, \mathbb{T})$. 
\end{definition}

Alternatively one can axiomatize representable tangent categories in terms of the infinitesimal object $D$, see \cite[Definition 5.6]{cockett2014differential}. Note that by definition, a representable functor has a left adjoint and therefore one can apply Theorem \ref{tanadj} to a representable tangent category. 

\begin{theorem}\label{tanadj2} \cite[Corollary 5.18]{cockett2014differential} Let $(\mathbb{X}, \mathbb{T})$ be a representable tangent category with infinitesimal object $D$. Then $\mathbb{X}^{op}$ has a tangent structure with tangent functor $- \times D$.   
\end{theorem} 

\begin{example} \normalfont\label{reptaneex} We finish this section with some examples of representable tangent categories. 
\begin{enumerate}
\item Every tangent category embeds into a representable tangent category \cite{garner2018embedding}. 
\item The subcategory of infinitesimally and vertically linear objects of any model of synthetic differential geometry \cite{kock2006synthetic} is a representable tangent category with infinitesimal object $D = \lbrace x \in R \vert ~ x^2 = 0 \rbrace$, where $R$ is the line object \cite[Proposition 5.10]{cockett2014differential}. 
\item \label{Crigop} Let $k$ be a field. Recall that in $\mathsf{CALG}_k$, the categorical coproduct is given by the tensor product of $k$-vector spaces $\otimes$ which is therefore a product in $\mathsf{CALG}_k^{op}$. Then $\mathsf{CALG}_k^{op}$ is a representable tangent category with infinitesimal object $k[\epsilon]$, the ring of dual numbers over $k$. For a commutative $k$-algebra $A$, $A^{k[\epsilon]}$ (in $\mathsf{CALG}_k^{op}$) is defined as the symmetric $A$-algebra over the K{\" a}hler module of $A$ (see \cite[Proposition 5.16]{cockett2014differential} for full details). By applying Theorem \ref{tanadj2} to this example, one obtains precisely the tangent structure on $\mathsf{CALG}_k$ from Example \ref{tanex}.\ref{Crigex}, where in particular we note that $A[\epsilon] \cong A \otimes k[\epsilon]$. In Section \ref{CoEMdiffsec} we will generalize this example to the context of differential categories. We again note that this example generalizes to the category of commutative algebras over any commutative semiring. 
\end{enumerate}
\end{example}

\section{Coalgebra Modalities and their coEilenberg-Moore Categories}\label{coalgref}

In this section we review (co)algebra modalities and take a look at their (co)Eilenberg-Moore categories. Coalgebra modalities were introduced in the development of differential categories \cite{blute2006differential}, as a weakening of the notion of \textit{linear exponential comonad} that is required for a categorical model of the multiplicative and exponential fragment of linear logic ($\mathsf{MELL}$) \cite{bierman1995categorical,schalk2004categorical,mellies2009categorical}.  While the notion of a coalgebra modality is strictly weaker than that of a linear exponential comonad---which is precisely a \emph{monoidal} coalgebra modality \cite{blute2015cartesian}---coalgebra modalities provide a sufficient context in which to axiomatize differentiation. 

A \textbf{comonad} \cite{mac2013categories} on a category $\mathbb{X}$ will be denoted as a triple $(\oc, \delta, \varepsilon)$ with endofunctor ${\oc: \mathbb{X} \to \mathbb{X}}$ and natural transformations $\delta_A: \oc(A) \to \oc \oc(A)$ and $\varepsilon_A: \oc(A) \to A$. A \textbf{$\oc$-coalgebra} will be denoted as a pair $(A, \omega)$ with underlying object $A$ and $\oc$-coalgebra structure ${\omega: A \to \oc(A)}$. The category of $\oc$-coalgebras and $\oc$-coalgebra morphisms is called the \textbf{coEilenberg-Moore category} \cite{mac2013categories} of the comonad $(\oc, \delta, \varepsilon)$ and will be denoted $\mathbb{X}^\oc$. Coalgebra modalities are comonads such that every cofree $\oc$-coalgebra comes equipped with a natural cocommutative comonoid structure. 

\begin{definition}\label{coalgdef} A \textbf{coalgebra modality} \cite{blute2006differential} on a symmetric monoidal category is a quintuple $(\oc, \delta, \varepsilon, \Delta, \mathsf{e})$ consisting of a comonad $(\oc, \delta, \varepsilon)$ equipped with two natural transformations $\Delta_A: \oc (A) \to \oc (A) \otimes \oc (A)$ and $\mathsf{e}_A: \oc (A) \to K$ such that for each object $A$, $(\oc(A), \Delta_A, \mathsf{e}_A)$ is a cocommutative comonoid and $\delta_A$ is a comonoid morphism. 
\end{definition}

What can we say about the coEilenberg-Moore category of a coalgebra modality? It turns out that every $\oc$-coalgebra of a coalgebra modality comes equipped with a cocommutative comonoid structure \cite{blute2015derivations}. Indeed if $(A, \omega)$ is a $\oc$-coalgebra, then the triple $(A, \Delta^\omega, \mathsf{e}^\omega)$ is a cocommutative comonoid where $\Delta^\omega$ and $\mathsf{e}^\omega$ are defined as follows: 
$$\Delta^\omega := \xymatrix{A \ar[r]^-{\omega} & \oc(A) \ar[r]^-{\Delta_A} & \oc(A) \otimes \oc(A) \ar[r]^-{\varepsilon_A \otimes \varepsilon_A} & A \otimes A  
  } \quad \quad \quad  \mathsf{e}^\omega :=   \xymatrix{A \ar[r]^-{\omega} & \oc(A) \ar[r]^-{\mathsf{e}_A} & K} $$
 It is important to point out that $(A, \Delta^\omega, \mathsf{e}^\omega)$ is in general only a cocommutative comonoid in the base category $\mathbb{X}$ and not in the coEilenberg-Moore category $\mathbb{X}^\oc$, since the latter does not necessarily have a monoidal product. Also notice that since $\delta_A$ is a comonoid morphism, when applying this construction to a cofree $\oc$-coalgebra $(\oc(A), \delta_A)$ we recover $\Delta_A$ and $\mathsf{e}_A$, that is, $\Delta^{\delta_A}=\Delta_A$ and $\mathsf{e}^{\delta_A}=\mathsf{e}_A$. 
 
\begin{definition}\label{Seelydef} In a symmetric monoidal category with finite products $\times$ and terminal object $\mathsf{1}$, a coalgebra modality has \textbf{Seely isomorphisms} \cite{bierman1995categorical,blute2015cartesian,seely1987linear} if the map $\chi_{\mathsf{1}}: \oc (\mathsf{1}) \to K$ and natural transformation ${\chi: \oc(A \times B) \to \oc A \otimes \oc B}$ defined respectively as
$$\xymatrixcolsep{2pc}\xymatrix{  \oc(\mathsf{1}) \ar[r]^-{e} & K &\oc(A \times B) \ar[r]^-{\Delta} & \oc(A \times B) \otimes \oc(A \times B) \ar[rr]^-{\oc(\pi_0) \otimes \oc(\pi_1)} && \oc (A) \otimes \oc (B)
  }$$
are isomorphisms, so $\oc(\mathsf{1}) \cong K$ and $\oc(A \times B) \cong \oc (A) \otimes \oc (B)$.  
\end{definition}

Coalgebra modalities with Seely isomorphisms can equivalently be defined as \textbf{monoidal coalgebra modalities} \cite{bierman1995categorical,blute2015cartesian}, which are coalgebra modalities equipped with a natural transformation $\mathsf{m}_{A,B}: \oc(A) \otimes \oc(B) \to \oc(A \otimes B)$ and a map $\mathsf{m}_K: K \to  \oc(K)$ making $\oc$ a symmetric monoidal comonad such that $\Delta$ and $\mathsf{e}$ are both monoidal transformations and $\oc$-coalgebra morphisms. Furthermore for a monoidal coalgebra modality, the monoidal product of the base category becomes a finite product in the coEilenberg-Moore category \cite{schalk2004categorical}. Explicitly, the terminal object is the $\oc$-coalgebra $(K, \mathsf{m}_K)$ while the product of $\oc$-coalgebras $(A, \omega)$ and $(B, \omega^\prime)$ is $(A, \omega) \otimes (B, \omega^\prime) := (A \otimes B, (\omega \otimes \omega^\prime)\mathsf{m}_{A,B})$. 

\begin{example} \normalfont\label{coalgmodex} Here are some examples of coalgebra modalities. Many other examples of coalgebra modalities (with and without the Seely isomorphisms) can be found in \cite{blute2018differential}. 
\begin{enumerate}
\item There is no shortage of examples of coalgebra modalities since every categorical model of $\mathsf{MELL}$ admits a coalgebra modality which has the Seely isomorphism. For example, Hyland and Schalk provide a nice list of such examples in \cite[Section 2.4]{hyland2003glueing}. 
\item \label{cocomex} Let $k$ be a field and let $\mathsf{VEC}_k$ be the category of $k$-vector spaces, which is a symmetric monoidal category with respect to the standard tensor product of $k$-vector spaces. For every $k$-vector space $V$, there exists a cofree cocommutative $k$-coalgebra \cite{sweedler1969hopf} over $V$, denoted $\oc(V)$, where a detailed construction can be found in \cite{clift2017cofree, hyland2003glueing, sweedler1969hopf}. In particular, if $k$ has characteristic $0$ \footnote{In \cite{clift2017cofree}, Clift and Murfet work, for simplicity, with an algebraically closed field of characteristic $0$.  However, as they point out, the assumption that the field is algebraically closed is not necessary in the construction.} and if $X=\lbrace x_i \mid i \in I\}$ is a basis of $V$, then  $\oc(V) \cong \bigoplus \limits_{v\in V} k[X]$ as $k$-coalgebras (where $k[X]$ is the polynomial ring over $k$ generated by the set $X$). This induces a coalgebra modality $\oc$ on $\mathsf{VEC}_k$ which furthermore has the Seely isomorphisms ($\oc(V \times W) \cong \oc(V) \otimes \oc(W)$ and $\oc(\mathsf{0}) \cong k$), and by \cite{porst2006corings} we know that the coEilenberg-Moore category of $\oc$ is isomorphic to the category of cocommutative $k$-coalgebras (which are the cocommutative comonoids in $\mathsf{VEC}_k$).  By applying results in \cite{porst2006corings}, one can generalize this example to the category of modules over an arbitrary commutative unital ring.
\end{enumerate}
\end{example}
 
The dual notion of a coalgebra modality is an algebra modality. Since we will be working with algebra modalities in Section \ref{EMCodiffsec}, we provide the definition of an algebra modality in detail. A \textbf{monad} \cite{mac2013categories}, the dual notion a comonad, on a category $\mathbb{X}$ will be denoted as a triple $(\mathsf{S}, \mu, \eta)$ consisting of an endofunctor $\mathsf{S}: \mathbb{X} \to \mathbb{X}$ and natural transformations $\mu_A: \mathsf{S}^2(A) \to \mathsf{S}(A)$ and $\eta_A: \mathsf{S}(A) \to A$. An $\mathsf{S}$-algebra will be denoted as a pair $(A, \nu)$ with underlying object $A$ and structure map $\nu: \mathsf{S}(A) \to A$. The category of $\mathsf{S}$-algebras and $\mathsf{S}$-algebra morphisms is called the \textbf{Eilenberg-Moore category} \cite{mac2013categories} of the monad $(\mathsf{S}, \mu, \eta)$ and is denoted $\mathbb{X}^\mathsf{S}$. 

  \begin{definition} An \textbf{algebra modality} \cite{blute2015derivations} on a symmetric monoidal category is a quintuple $(\mathsf{S}, \mu, \eta,$ $\nabla, \mathsf{u})$ consisting of a monad $(\mathsf{S}, \mu, \eta)$ equipped with two natural transformations $\nabla_A: \mathsf{S}(A) \otimes \mathsf{S}(A) \to \mathsf{S}(A)$ and $\mathsf{e}: K \to \mathsf{S}(A)$ such that for each object $A$, $(\mathsf{S}(A), \nabla_A, \mathsf{u}_A)$ is a commutative monoid and $\mu_A$ is a monoid morphism. 
\end{definition}

Since algebra modalities are dual to coalgebra modalities, it follows that every $\mathsf{S}$-algebra comes equipped with a commutative monoid structure \cite{blute2015derivations}, which we again point out is in the base category and not the Eilenberg-Moore category. Explicitly, given an $\mathsf{S}$-algebra $(A, \nu)$ of an algebra modality $(\mathsf{S}, \mu, \eta, \nabla, \mathsf{u})$, the triple $(A, \nabla^\nu, \mathsf{u}^\nu)$ is a commutative monoid where $\nabla^\omega$ and $\mathsf{u}^\omega$ are defined as follows: 
$$ \nabla^\nu := \xymatrixcolsep{2pc}\xymatrix{A \otimes A \ar[r]^-{\eta_A \otimes \eta_A} & \mathsf{S}(A) \otimes \mathsf{S}(A) \ar[r]^-{\nabla_A} & \mathsf{S}(A) \ar[r]^-{\nu_A} & A}  \quad \quad \mathsf{u}^\nu :=\xymatrixcolsep{2pc}\xymatrix{K \ar[r]^-{\mathsf{u}_A} & \mathsf{S}(A) \ar[r]^-{\nu} & A
  } $$ 
Dual to coalgebra modalities with the Seely isomorphisms, in the case of a symmetric monoidal category with finite coproducts $\oplus$ and initial object $\mathsf{0}$, if the algebra modality has the Seely isomorphisms, i.e. $\mathsf{S}(\mathsf{0}) \cong K$ and $\mathsf{S}(A \oplus B) \cong \mathsf{S}(A) \otimes \mathsf{S}(B)$, then the monoidal product becomes a coproduct in Eilenberg-Moore category. 

\begin{example} \normalfont\label{algmodex} Here are some examples of algebra modalities. Many other examples of algebra modalities (with and without the Seely isomorphisms) can be found in \cite{blute2018differential}. 
\begin{enumerate}
\item \label{Symex} Let $k$ be a field. In analogy with Example \ref{coalgmodex}.\ref{cocomex}, for every $k$-vector space $V$ there exists a free commutative $k$-algebra over $V$, denoted $\mathsf{Sym}(V)$, which is also called the symmetric algebra over $V$ \cite{lang2002algebra}. In particular, if $X=\lbrace x_i \mid i \in I\}$ is a basis of $V$, then  $\mathsf{Sym}(V) \cong k[X]$ as $k$-algebras (where $k[X]$ is the polynomial $k$-algebra generated by the set $X$). This induces an algebra modality $\mathsf{Sym}$ on $\mathsf{VEC}_k$ which furthermore has the Seely isomorphisms (so that $\mathsf{Sym}(V \times W) \cong \mathsf{Sym}(V) \otimes \mathsf{Sym}(W)$ and $\mathsf{Sym}(\mathsf{0}) \cong k$) and whose Eilenberg-Moore category is isomorphic to the category of commutative $k$-algebras (which are the commutative monoids in $\mathsf{VEC}_k$). This example generalizes to the category of modules over an arbitrary commutative unital semiring.  
\item \label{Cex} Let $\mathbb{R}$ be the field of real numbers. $C^\infty$-rings \cite{kock2006synthetic,moerdijk2013models} are defined as the algebras of the Lawvere theory whose morphisms are smooth maps between Cartesian spaces $\mathbb{R}^n$, so a $C^\infty$-ring can be defined equivalently as a set $A$ equipped with a family of functions $\Phi_f: A^n \to A$ indexed by the smooth functions $f: \mathbb{R}^n \to \mathbb{R}$, satisfying certain equations. For example, if $M$ is a smooth manifold, then $C^\infty(M) = {\lbrace f: M \to \mathbb{R} \vert~ f \text{ is smooth} \rbrace}$ is a $C^\infty$-ring where for a smooth map $f: \mathbb{R}^n \to \mathbb{R}$, $\Phi_f: C^\infty(M)^n  \to C^\infty(M)$ is defined by post-composition by $f$. Every $C^\infty$-ring is a commutative $\mathbb{R}$-algebra and for every $\mathbb{R}$-vector space $V$ there exists a free $C^\infty$-ring over $V$ \cite[Theorem 3.3]{kainz1987c}\cite{cruttwell2019integral}, denoted as $\mathsf{S}^\infty(V)$. This induces an algebra modality $\mathsf{S}^\infty$ on $\mathsf{VEC}_\mathbb{R}$ \cite{cruttwell2019integral}, where in particular for a finite dimensional vector space $V$ of dimension $n$, one has that $\mathsf{S}^\infty(V) \cong C^\infty(\mathbb{R}^n)$. The Eilenberg-Moore category of $\mathsf{S}^\infty$ is the category of $C^\infty$-rings \cite{kock2006synthetic,moerdijk2013models,cruttwell2019integral}. It is important to note that this is an example of an algebra modality which does NOT have the Seely isomorphisms. However, $C^\infty$-rings are mathematically important, as they provide the basis for well-adapted models of synthetic differential geometry \cite{kock2006synthetic,moerdijk2013models} and provide a natural setting for the adaptation of algebro-geometric methods to a smooth context. 
\end{enumerate}
\end{example}

\section{Differential Categories}\label{diffsec}

In this section we review (co)differential categories and, in particular, we will take a look at some well-known examples that correspond to differentiating polynomials and smooth maps. Differential categories were introduced by Blute, Cockett, and Seely \cite{blute2006differential} to provide the categorical semantics of differential linear logic \cite{ehrhard2017introduction}. While a codifferential category is simply the dual of a differential category, we provide full definitions of each since we will be working with codifferential categories in Section \ref{EMCodiffsec} and in the appendix, and we will be working with differential categories in Section \ref{CoEMdiffsec}. 

Two of the basic properties of the derivative from classical differential calculus require addition (or at least the number $0$): the Leibniz rule and the constant rule. Therefore we must first discuss additive structure. Here we mean ``additive'' in the sense of \cite{blute2006differential}, that is, enriched over commutative monoids. In particular this definition does not assume negatives nor does it assume biproducts, so this differs from other definitions of an additive category such as in \cite{mac2013categories}. That said, in Section \ref{EMCodiffsec} and Section \ref{CoEMdiffsec} we will be working with (co)differential categories with biproducts. 

\begin{definition}\label{AC} An \textbf{additive category} is a category enriched in commutative monoids, that is, a category in which each hom-set is a commutative monoid, with addition operation $+$ and zero $0$, and in which composition preserves the additive structure in the sense that $k(f+g)h=(kfh)+(kgh)$ and $0f=0=f0$. An \textbf{additive symmetric monoidal category} \cite{blute2006differential} is a symmetric monoidal category that is also an additive category such that the monoidal product $\otimes$ is compatible with the additive structure in the sense that $k \otimes (f+g)\otimes h=  k \otimes f\otimes h + k \otimes g\otimes h$ and $f \otimes 0\otimes g=0$. 
\end{definition}

Every category with finite biproducts is an additive category, and
finite (co)products in an additive category are automatically finite biproducts. In fact, every additive category can be completed to a category with biproducts \cite{mac2013categories}, where the completion is also an additive category, and similarly every additive symmetric monoidal category can be completed to an additive symmetric monoidal category with distributive biproducts. For this reason it is possible to argue (as in \cite{fiore2007differential}) that one might as well always assume one has biproducts.  However, it is important to bear in mind that only monoidal coalgebra modalities are guaranteed to lift to the biproduct completion.  Thus, for a treatment of arbitrary coalgebra modalities this assumption cannot be made (see \cite{blute2018differential} for more details). 

\begin{definition}\label{diffdef} A \textbf{differential category} \cite{blute2006differential} is an additive symmetric monoidal category with a coalgebra modality $(\oc, \delta, \varepsilon, \Delta, \mathsf{e})$ that comes equipped with a \textbf{deriving transformation}, that is, a natural transformation $\mathsf{d}_A: \oc (A) \otimes A \to \oc(A)$ such that the following equalities hold: \\
\textbf{[d.1]} Constant Rule: $\mathsf{d}_A\mathsf{e}_A= 0$ \\
\textbf{[d.2]} Leibniz Rule: $\mathsf{d}_A\Delta_A =(\Delta_A \otimes 1_{\oc(A)})(1_{\oc(A)} \otimes \tau_{\oc(A),A})(\mathsf{d}_A \otimes 1_{\oc(A)}) + (\Delta_A \otimes 1_{\oc(A)})(1_{\oc(A)} \otimes \mathsf{d}_A)$ \\
\textbf{[d.3]} Linear Rule: $\mathsf{d}_A\varepsilon_A=\mathsf{e}_A \otimes 1_A$ \\
\textbf{[d.4]} Chain Rule: $\mathsf{d}_A\delta_A=(\Delta_A \otimes 1_A)(\delta_A \otimes \mathsf{d}_A)\mathsf{d}_{\oc(A)}$ \\
\textbf{[d.5]} Interchange Rule\footnote{It should be noted that the interchange rule {\bf [d.5]} was not part of the definition in \cite{blute2006differential} but was later added to ensure that the coKleisli category of a differential category (with finite products) was a Cartesian differential category \cite{blute2009cartesian}.}: $(\mathsf{d}_A \otimes 1_A)\mathsf{d}_A=(1 \otimes \tau_{A,A})(\mathsf{d}_A \otimes 1_A)\mathsf{d}_A$ 
\end{definition}

CoKleisli maps of coalgebra modalities, that is, maps of type $f: \oc(A) \to B$, are to be thought of as \emph{smooth} maps from $A$ to $B$ as they are, in a certain sense, infinitely differentiable. Indeed the \textit{derivative} of $f: \oc(A) \to B$ is the map $\mathsf{D}[f]: \oc (A) \otimes A \to B$, defined as the composite $\mathsf{D}[f]:= \mathsf{d}_A f$. The constant rule {\bf [d.1]} amounts to the statement that the derivative of a constant map is zero. The second axiom {\bf [d.2]} is the analogue of the classical Leibniz rule in differential calculus. For the third axiom, a subclass of smooth maps are the \emph{linear} maps, which are coKleisli maps of the form $\varepsilon_A g: \oc (A) \to B$ for some map $g: A \to B$. Then the linear rule {\bf [d.3]} says that the derivative of a linear map is `constant' with respect to the point at which it is taken.  The fourth axiom {\bf [d.4]} is the chain rule regarding composition in the coKleisli category. The interchange rule {\bf [d.5]}, is the independence of order of differentiation, which, naively put, says that differentiating with respect to $x$ then $y$ is the same as differentiation with respect to $y$ then $x$. For more details and for string diagram representation of the axioms of a differential category, we refer the reader to \cite{blute2006differential, blute2018differential}. 

\begin{definition} A \textbf{differential storage category} \cite{blute2006differential} is a differential category with finite biproducts whose coalgebra modality has the Seely isomorphisms. 
\end{definition}

For differential storage categories, the differential category structure can equivalently be axiomatized by a natural transformation ${\eta_A: A \to \oc(A)}$ known as a \textbf{codereliction} \cite{blute2006differential,blute2018differential,fiore2007differential}. For monoidal coalgebra modalities, there is a bijective correspondence between coderelictions and deriving transformations \cite{blute2018differential}. However, we will see below that the deriving transformation plays the more important role when discussing tangent category structure of (co-)Eilenberg-Moore categories of differential categories. 

\begin{example} \normalfont\label{diffex} Here are some examples of differential storage categories. Many other examples of differential (storage) categories can be found in \cite{blute2006differential,blute2018differential,ehrhard2017introduction}. 
\begin{enumerate}
\item In \cite{blute2010convenient}, Blute, Ehrhard, and Tasson showed that the category of convenient vector spaces and bounded linear maps between them is a differential storage category. In particular, the coKleisli category is precisely the category of convenient vector spaces and \emph{smooth} maps between them. The differential structure is induced by the limit definition of the derivative in locally convex vector spaces. 
\item  In \cite{ehrhard2017introduction}, Ehrhard provides a differential storage category (amongst others) whose objects are \textit{linearly topologized} vector spaces generated by finiteness spaces. The differential structure corresponds to differentiating multivariable power series. 
\item \label{cofreediffex} In \cite{clift2017cofree}, Clift and Murfet study the differential category structure induced by cofree cocommutative coalgebras. So if $k$ is a field of characteristic $0$, $\mathsf{VEC}_k$ is a differential storage category with the coalgebra modality $\oc$ defined in Example \ref{coalgmodex}.\ref{cocomex}. Recalling that for a vector space $V$ with basis $X$, $\oc(V) \cong \bigoplus \limits_{v\in V} k[X]$, the deriving transformation can be expressed as:
$$\mathsf{d}_V: \left(\bigoplus \limits_{v\in V} k[X] \right) \otimes V \to \bigoplus \limits_{v\in V} k[X]  \quad \quad   p_v(x_1, \hdots, x_n) \otimes x_i \mapsto p_v(x_1, \hdots, x_n)x_i $$
where $p_v(x_1, \hdots, x_n)$ is a polynomial in distinct indeterminates $x_1,x_2,...,x_n \in X$ and lies in the $v$-th coproduct-component of $\oc V$, where $v \in V$ . This example also generalizes for modules over a commutative unital semiring. 
\end{enumerate}
\end{example}

\begin{definition}\label{codiffdef} A \textbf{codifferential category} \cite{blute2015derivations} is an additive symmetric monoidal category with an algebra modality $(\mathsf{S}, \mu, \eta, \nabla, \mathsf{u})$ that comes equipped with a \textbf{deriving transformation}\footnote{As in the literature, we keep the same terminology and notation for a deriving transformation in the context of a codifferential category}, that is, a natural transformation $\mathsf{d}_A: \mathsf{S}(A) \to \mathsf{S}(A) \otimes A$ such that dual equalities of \textbf{[d.1]} to \textbf{[d.5]} hold, that is: \\
\textbf{[cd.1]} Constant Rule: $\mathsf{u}_A\mathsf{d}_A= 0$ \\
\textbf{[cd.2]} Leibniz Rule: 
$$\nabla_A\mathsf{d}_A =(\mathsf{d}_A \otimes 1_{\mathsf{S}(A)})(1_{\mathsf{S}(A)} \otimes \tau_{A,\mathsf{S}(A)})(\nabla_A \otimes 1_{\mathsf{S}(A)})+ (1_{\mathsf{S}(A)} \otimes \mathsf{d}_A)(\nabla_A \otimes 1_{\mathsf{S}(A)})$$
\textbf{[cd.3]} Linear Rule: $\eta_A\mathsf{d}_A=\mathsf{u}_A \otimes 1_A$ \\
\textbf{[cd.4]} Chain Rule: $\mu_A\mathsf{d}_A=\mathsf{d}_{\mathsf{S}(A)}(\mu_A \otimes \mathsf{d}_A)(\nabla_A \otimes 1_A)$ \\
\textbf{[cd.5]} Interchange Rule: $\mathsf{d}_A(\mathsf{d}_A \otimes 1_A)=\mathsf{d}_A(\mathsf{d}_A \otimes 1_A)(1 \otimes \tau_{A,A})$ 
\end{definition}

\begin{example} \normalfont\label{codiffex} Here are some examples of codifferential categories. Many other examples of codifferential categories can be found in \cite{blute2006differential,blute2018differential}. 
 \begin{enumerate}
\item \label{MODdiffex} Let $k$ be a field. Then $\mathsf{VEC}_k$ is a codifferential storage category with the algebra modality $\mathsf{Sym}$ defined in Example \ref{algmodex}.\ref{Symex}, and where the differential structure corresponds precisely to differentiation of polynomials. To see this, recall that for a $k$-vector space $V$ with basis set $X$, $\mathsf{Sym}(V) \cong k[X]$. Therefore the deriving transformation can be expressed as $\mathsf{d}_V: k[X] \to k[X] \otimes V$, which is given by taking a sum involving the partial derivatives:
$$\mathsf{d}_V: k[X] \to k[X] \otimes V  \quad \quad \quad p(x_1, \hdots, x_n) \mapsto \sum_{i=1}^{n} \frac{\partial p}{\partial x_i}(x_1, \hdots, x_n) \otimes x_i $$
So $\mathsf{VEC}^{op}_k$ is a differential storage category. See \cite{blute2006differential,blute2018differential} for full details on this example. This example generalizes to the category of modules over a commutative unital semiring. 
\item \label{Cinfdiffex} Other than the codifferential category structure given by $\mathsf{Sym}$ from the previous example, $\mathsf{VEC}_\mathbb{R}$ also has a codifferential category structure with respect to the algebra modality $\mathsf{S}^\infty$ defined in Example \ref{algmodex}.\ref{Cex}. The deriving transformation is induced by differentiating smooth functions.  In particular for $\mathbb{R}^n$, $\mathsf{S}^\infty(\mathbb{R}^n)= C^\infty(\mathbb{R}^n)$ and the deriving transformation $\mathsf{d}_{\mathbb{R}^n}: C^\infty(\mathbb{R}^n) \to C^\infty(\mathbb{R}^n) \otimes \mathbb{R}^n$ is defined as a sum involving the partial derivatives: 
$$\mathsf{d}_{\mathbb{R}^n}: C^\infty(\mathbb{R}^n) \to C^\infty(\mathbb{R}^n) \otimes \mathbb{R}^n  \quad \quad \quad f \mapsto \sum_{i=1}^{n} \frac{\partial f}{\partial x_i} \otimes x_i $$
Hence $\mathsf{VEC}^{op}_\mathbb{R}$ is a differential category. See \cite{cruttwell2019integral} for full details on this example.
\end{enumerate}
\end{example}

\section{Tangent Structure and Codifferential Categories} \label{EMCodiffsec}

The goal of this section is to prove that the Eilenberg-Moore category of a codifferential category with finite biproducts is a tangent category. To achieve this we need to first introduce the concept of a \emph{tangent monad}, in order to lift tangent structure to Eilenberg-Moore categories. 

Let $(\mathsf{S}, \mu, \eta)$ be a monad on a category $\mathbb{X}$ and let $\mathsf{T}: \mathbb{X} \to \mathbb{X}$ be an endofunctor. Recall that a \textbf{distributive law} \cite{wisbauer2008algebras} of $\mathsf{T}: \mathbb{X} \to \mathbb{X}$ over $(\mathsf{S}, \mu, \eta)$ is a natural transformation $\lambda_M: \mathsf{S}(\mathsf{T}(M)) \to \mathsf{T}(\mathsf{S}(M))$ such that the following diagrams commute:  
$$ \xymatrixcolsep{1.5pc}\xymatrixrowsep{2pc}\xymatrix{\mathsf{S}^2\mathsf{T}(M) \ar[d]_-{\mu_{\mathsf{T}(M)}} \ar[r]^-{\mathsf{S}(\lambda_M)} &   \mathsf{S}\mathsf{T}\mathsf{S}(M)  \ar[r]^-{\lambda_{\mathsf{S}(M)}} &   \mathsf{T}\mathsf{S}^2(M) \ar[d]^-{\mathsf{T}(\mu_M)} & \mathsf{T}(M)\ar[dr]_-{\mathsf{T}(\eta_M)} \ar[r]^-{\eta_{\mathsf{T}(M)}} & \mathsf{S}\mathsf{T}(M) \ar[d]^-{\lambda_M} \\
\mathsf{S}\mathsf{T}(M) \ar[rr]_-{\lambda_M} && \mathsf{T}\mathsf{S}(M) && \mathsf{T}\mathsf{S}(M) 
  } $$
Distributive laws of this sort allow us to lift $\mathsf{T}$ to the Eilenberg-Moore category of $(\mathsf{S}, \mu, \eta)$ \cite{wisbauer2008algebras}, noting that this is an instance of a more general result of Appelgate that is stated in \cite{johnstone1975adjoint_lifting}. Explicitly, the endofunctor $\overline{\mathsf{T}}: \mathbb{X}^\mathsf{S} \to \mathbb{X}^\mathsf{S}$, called the lifting of $\mathsf{T}$, is defined on objects by $\overline{\mathsf{T}}(A, \xymatrixcolsep{1pc}\xymatrix{\mathsf{S}(A) \ar[r]^-{\nu} & A}) := (\mathsf{T}(A), \xymatrixcolsep{1pc}\xymatrix{\mathsf{S} \mathsf{T}(A) \ar[r]^-{\lambda_A} & \mathsf{T}\mathsf{S}(A) \ar[r]^-{\mathsf{T}(\nu)} & \mathsf{T}(A)} \!\!)$ and on maps by  
$\overline{\mathsf{T}}(f) := \mathsf{T}(f)$. 
\begin{definition}\label{tanmonaddef} A \textbf{tangent monad} on a tangent category $(\mathbb{X}, \mathbb{T})$ is a quadruple $(\mathsf{S}, \eta, \mu, \lambda)$ consisting of a monad $(\mathsf{S}, \eta, \mu)$ equipped with a distributive law $\lambda$ of the tangent functor $\mathsf{T}$ over $(\mathsf{S}, \mu, \delta)$ such that the following diagrams commute: 
$$ \xymatrixcolsep{1.5pc}\xymatrixrowsep{2pc}\xymatrix{ & \mathsf{S}\mathsf{T}(M) \ar[d]_-{\lambda_M} \ar[dr]^-{\mathsf{S}(p_M)} && \mathsf{S}\mathsf{T}_2(M) \ar[r]^-{\mathsf{S}(\sigma_M)} \ar[d]|-{\langle \mathsf{S}(\rho_0)\lambda, \mathsf{S}(\rho_1)\lambda \rangle } & \mathsf{S}\mathsf{T}(M) \ar[d]^-{\lambda_M} & \mathsf{S}(M) \ar[dr]_-{z_{\mathsf{S}(M)}} \ar[r]^-{\mathsf{S}(z_M)} & \mathsf{S}\mathsf{T}(M) \ar[d]^-{\lambda_M} \\
& \mathsf{T}\mathsf{S}(M) \ar[r]_-{p_{\mathsf{S}(M)}} & \mathsf{S}(M) & \mathsf{T}_2\mathsf{S}(M) \ar[r]_-{\sigma_{\mathsf{S}(M)}} & \mathsf{T}\mathsf{S}(M) & & \mathsf{T}\mathsf{S}(M)} $$
 $$ \xymatrixcolsep{1.5pc}\xymatrixrowsep{2pc}\xymatrix{ \mathsf{S}\mathsf{T}(M) \ar[dd]_-{\lambda_M} \ar[r]^-{\mathsf{S}(\ell_M)} & \mathsf{S}\mathsf{T}^2(M) \ar[d]^-{\lambda_{\mathsf{T}(M)}} && \mathsf{S}\mathsf{T}^2(M) \ar[d]_-{\lambda_{\mathsf{T}(M)}} \ar[r]^-{\mathsf{S}(c_M)} & \mathsf{S}\mathsf{T}^2(M) \ar[d]^-{\lambda_{\mathsf{T}(M)}} \\
 & \mathsf{T}\mathsf{S}\mathsf{T}(M) \ar[d]^-{\mathsf{T}(\lambda_M)} && \mathsf{T}\mathsf{S}\mathsf{T}(M) \ar[d]_-{\mathsf{T}(\lambda_M)} & \mathsf{T}\mathsf{S}\mathsf{T}(M) \ar[d]^-{\mathsf{T}(\lambda_M)}   \\
\mathsf{T}\mathsf{S}(M) \ar[r]_-{\ell_{\mathsf{S}(M)}} & \mathsf{T}^2\mathsf{S}(M) && \mathsf{T}^2\mathsf{S}(M) \ar[r]_-{c_{\mathsf{S}(M)}} & \mathsf{T}^2\mathsf{S}(M)
  }$$
  \end{definition}
 
Equivalently a tangent monad is a monad in the 2-category of tangent categories \cite{garner2018embedding} since the above diagrams imply that $(\mathsf{S}, \lambda)$ is a tangent morphism \cite{cockett2014differential}. 

\begin{proposition}\label{tanlift} The Eilenberg-Moore category of a tangent monad is a tangent category such that the forgetful functor preserves the tangent structure strictly. 
\end{proposition} 
\begin{proof} Let $(\mathsf{S}, \eta, \mu, \lambda)$ be a tangent monad on a tangent category $(\mathbb{X}, \mathbb{T})$. Since $\lambda$ is a distributive law, this induces a lifting functor $\overline{\mathsf{T}}: \mathbb{X}^\mathsf{S} \to \mathbb{X}^\mathsf{S}$. Then we can define $\overline{p}_{(A,\nu)} := p_A$, $\overline{z}_{(A,\nu)} := z_A$, $\overline{\ell}_{(A,\nu)} := \ell_A$, $\overline{c}_{(A,\nu)} := c_A$. That $\overline{p}$, $\overline{z}$, $\overline{\ell}$, and $\overline{c}$ are all $\mathsf{S}$-algebra morphisms follows from naturality of $p,z,\ell,c$ and the respective diagrams of a tangent monad. To define the addition map $\overline{\sigma}$, we must first address limits in $\mathbb{X}^\mathsf{S}$. It is well known that any given diagram has a limit in the Eilenberg-Moore category as soon as it has a limit in the base category \cite{mac2013categories}. Therefore the tangent limits of $(\mathbb{X}, \mathbb{T})$ easily lift to $\mathbb{X}^\mathsf{S}$, where in particular
$ \overline{\mathsf{T}}_2(A, \xymatrixcolsep{1.5pc}\xymatrix{\mathsf{S}(A) \ar[r]^-{\nu} & A}) :=  (\mathsf{T}_2(A), \xymatrixcolsep{4.75pc}\xymatrix{\mathsf{S}\mathsf{T}_2(A) \ar[r]^-{\langle \mathsf{S}(\rho_0)\lambda, \mathsf{S}(\rho_1)\lambda \rangle} & \mathsf{T}_2\mathsf{S}(A) \ar[r]^-{\mathsf{T}_2(\nu)} & \mathsf{T}_2(A))}$. 
Then we have that $\overline{\sigma}_{(A,\nu)} := \sigma_A$. It follows that $\overline{\mathbb{T}} := (\overline{\mathsf{T}}, \overline{p}, \overline{\sigma}, \overline{z}, \overline{\ell}, \overline{c})$ is a tangent structure on $\mathbb{X}^\mathsf{S}$, which by definition is preserved strictly by the forgetful functor. We conclude that $(\mathbb{X}^\mathsf{S}, \overline{\mathbb{T}})$ is a tangent category. 
\end{proof} 

The converse of Proposition \ref{tanlift} is also true, that is, if the Eilenberg-Moore category of a monad admits a tangent structure that is strictly preserved by the forgetful functor, then said monad is a tangent monad. In fact, in analogy with results for other kinds of distributive laws \cite{wisbauer2008algebras}, tangent monads are in bijective correspondence with liftings of the tangent structure in this sense. Also note that by the universal property of the pullback there are distributive laws $\lambda_{n,M}: \mathsf{S}\mathsf{T}_n (M) \to \mathsf{T}_n\mathsf{S}(M)$ for each $n \in \mathbb{N}$. 

To provide a tangent structure on the Eilenberg-Moore category of a codifferential category, we will define a tangent monad structure on the algebra modality itself. However we first need to address which tangent structure of the base category we will be lifting to the Eilenberg-Moore category. This is where finite biproducts come into play. Recall that a category with \textbf{finite biproducts} \cite{mac2013categories} can be described as an additive category with a zero object $0 = 1$ such that for each pair of objects $A$ and $B$, there is an object $A \oplus B$ and maps $\iota_0: A \to A \oplus B$, $\iota_1: B \to A \oplus B$, $\pi_0: A \oplus B \to A$, and $\pi_1: A \oplus B \to B$, satisfying the well-known identities. This makes $A \oplus B$ both a product and a coproduct of $A$ and $B$. 

Every category $\mathbb{X}$ with finite biproducts admits a tangent structure whose tangent functor is the diagonal functor, that is, the tangent functor $\mathsf{T}$ is defined on objects as $\mathsf{T}(A) := A \oplus A$ and on maps as $\mathsf{T}(f) := f \oplus f$. The projection is $p_A := \xymatrixcolsep{1.5pc}\xymatrix{A \oplus A \ar[r]^-{\pi_0} & A }$, the zero is $z_A := \xymatrixcolsep{1.5pc}\xymatrix{A \ar[r]^-{\iota_0} & A \oplus A }$, the vertical lift is $\ell_A := \xymatrixcolsep{2pc}\xymatrix{A \oplus A \ar[r]^-{\iota_0 \oplus \iota_1} & A \oplus A \oplus A \oplus A }$, and the canonical flip is $ c_A := \xymatrixcolsep{3pc}\xymatrix{A \oplus A \oplus A \oplus A \ar[r]^-{1 \oplus \tau^\oplus \oplus 1} & A \oplus A \oplus A \oplus A}$ where $\tau^\oplus_{A,B}: A \oplus B \cong B \oplus A$ is the canonical symmetry isomorphism of the biproduct. Therefore it follows that for every $n \in \mathbb{N}$, $\mathsf{T}_n(A) := \bigoplus\limits^n_{i=0} A$. For $n=2$, $\mathsf{T}_2(A) := A \oplus A \oplus A$ and the addition map is $ \sigma_A := \xymatrixcolsep{3pc}\xymatrix{A \oplus A \oplus A \ar[r]^-{1 \oplus (\pi_0 + \pi_1)} & A \oplus A }$. We denote this tangent structure as $\mathbb{B}$ (for biproduct). That $(\mathbb{X}, \mathbb{B})$ is a tangent category follows from that fact that every category with finite biproducts is in fact a Cartesian differential category (see Example \ref{tanex}.\ref{Cartex}).  

\begin{proposition}\label{Codifflambda} Let $\mathbb{X}$ be a codifferential category with algebra modality $(\mathsf{S}, \eta, \mu, \nabla, \mathsf{u})$ and deriving transformation $\mathsf{d}$, and suppose that $\mathbb{X}$ admits finite biproducts $\oplus$. Define the natural transformation $\lambda_A: \mathsf{S}(A \oplus A) \to \mathsf{S}(A) \oplus \mathsf{S}(A)$ as the unique map that makes the following diagram commute: 
  \[  \xymatrixcolsep{2.25pc}\xymatrixrowsep{0.5pc}\xymatrix{ & \mathsf{S}(A \oplus A) \ar[dddl]_-{\mathsf{S}(\pi_0)} \ar@{-->}[ddd]^-{\lambda_A} \ar[r]^-{\mathsf{d}} & \mathsf{S}(A \oplus A) \otimes (A \oplus A) \ar[r]^-{\mathsf{S}(\pi_0) \otimes \pi_1} & \mathsf{S}(A) \otimes A \ar[r]^-{1_{\mathsf{S}(A)} \otimes \eta_A} & \mathsf{S}(A) \otimes \mathsf{S}(A) \ar[ddd]^-{\nabla_A} &  \\
  &&  \\
 && \\
\mathsf{S}(A) & \ar[l]^-{\pi_0} \mathsf{S}(A) \oplus \mathsf{S}(A) \ar[rrr]_-{\pi_1} &&& \mathsf{S}(A)  
  } \]
Equivalently, using the additive structure, $\lambda := \mathsf{S}(\pi_0) \iota_0 + \mathsf{d}(\mathsf{S}(\pi_0) \otimes \pi_1)(1 \otimes \eta)\nabla \iota_1$. 
Then $(\mathsf{S}, \mu, \eta, \lambda)$ is a tangent monad on $(\mathbb{X}, \mathbb{B})$. 
\end{proposition} 
\begin{proof} See Appendix. 
\end{proof} 

An immediate consequence of Proposition \ref{tanlift} and Proposition \ref{Codifflambda} is that $(\mathbb{X}^\mathsf{S}, \overline{\mathbb{B}})$ is a tangent category. Summarizing, we obtain one of the main results of this paper: 

\begin{theorem}\label{thmcodiff} The Eilenberg-Moore category of a codifferential category with finite biproducts is a tangent category. 
\end{theorem} 

In particular, the result of applying the tangent functor of $(\mathbb{X}^\mathsf{S}, \overline{\mathbb{B}})$ to an $\mathsf{S}$-algebra can be simplified to $\overline{\mathsf{T}}(A, \nu) := (A \oplus A, \mathsf{S}(\pi_0) \nu \iota_0+ \mathsf{d}(\mathsf{S}(\pi_0) \otimes \pi_1)(\nu \otimes 1)\nabla^\nu \iota_1)$, 
which is an instance of the $\mathsf{S}$-algebra structure defined in \cite[Theorem 4.1]{blute2015derivations}. Denote the $\mathsf{S}$-algebra structure of $\overline{\mathsf{T}}(A, \nu)$ as ${\nu^\flat: \mathsf{S}(A \oplus A) \to A \oplus A}$. 
By \cite[Proposition 5.4]{blute2015derivations}, the induced commutative monoid structure on $\overline{\mathsf{T}}(A, \nu)$, generalizes that of the ring of dual numbers from Example \ref{tanex}.\ref{Crigex}:
$$\nabla^{\nu^\flat} = (\pi_0 \otimes \pi_0)\nabla^\nu \iota_0 + \left[ (\pi_0 \otimes \pi_1) + (\pi_1 \otimes \pi_0) \right] \nabla^\nu \iota_1 \quad \quad \mathsf{u}^{\nu^\flat} = \mathsf{u}^\nu \iota_0 $$
Thus, the above tangent structure on the Eilenberg-Moore category of a codifferential category further highlights the relation between tangent structure and Weil algebras \cite{leung2017classifying}. 

\begin{example} \normalfont We conclude this section with some of the resulting tangent categories from our main examples of codifferential categories: 
 \begin{enumerate}
\item For a field $k$, when applying the constructions of this section to Example \ref{codiffex}.\ref{MODdiffex}, one recovers precisely the tangent category from Example \ref{tanex}.\ref{Crigex} induced by dual numbers, recalling that ${\mathsf{VEC}^{\mathsf{Sym}}_k \cong \mathsf{CALG}_k}$.
\item For Example \ref{codiffex}.\ref{Cinfdiffex}, the resulting tangent structure on the category of $C^\infty$-rings is particularly important, since the ring of dual numbers plays a key role in models of synthetic differential geometry based on $C^\infty$-rings \cite{kock2006synthetic,moerdijk2013models}. 
\item For a field $k$, one can also apply the constructions of this section to Example \ref{diffex}.\ref{cofreediffex}, which implies that the opposite category of cocommutative $k$-coalgebras is a tangent category.  
\end{enumerate}
\end{example}

\section{Representable Tangent Structure and Differential Categories} \label{CoEMdiffsec}

The goal of this section is to show that the coEilenberg-Moore category of a differential category with biproducts is a tangent category, provided that it has certain equalizers. We will also explain that in the case of a differential storage category, the coEilenberg-Moore category is in fact a \emph{representable} tangent category. To achieve this, we wish to apply Theorem \ref{tanadj} to the tangent structure that we constructed in the previous section, which resides on the Eilenberg-Moore category of a codifferential category. 

In a category $\mathbb{X}$ with finite biproducts, the pullback powers of the projection $p_A := \pi_0$ are $\mathsf{T}_n(A) = \bigoplus\limits^n_{i=0} A$. By the universality of the product and couniversality of the coproduct, each $\mathsf{T}_n$ is its own adjoint, that is, $\mathsf{T}_n$ is a left adjoint to $\mathsf{T}_n$. However in the Eilenberg-Moore category, $\overline{\mathsf{T}}_n$ is not necessarily its own adjoint (in fact it is far from it in any of the examples in this paper). Therefore we cannot use results about lifting adjunctions to Eilenberg-Moore category on the nose such as in \cite{hermida1998structural}. Instead we employ Johnstone's left adjoint lifting theorem \cite[Theorem 2]{johnstone1975adjoint_lifting}, which is a special case of the adjoint lifting theorem of Butler that can be found in \cite[Theorem 7.4]{barr2000toposes}, and it is at this point that we require the mild further assumption that the Eilenberg-Moore category admits reflexive coequalizers. First recall that a reflexive pair is a pair of parallel maps $f, g: A \to B$ with a common section, that is, there is a map $h: B \to A$ such that $hf=1_B=hg$. A reflexive coequalizer is a coequalizer of a reflexive pair, and a category is said to have reflexive coequalizers if it has coequalizers of all reflexive pairs. A famous result of Linton's is that for a monad on a cocomplete category, the Eilenberg-Moore category is cocomplete if and only if the Eilenberg-Moore category admits all reflexive coequalizers \cite{linton1969coequalizers}.  
%Instead we require a specialized version of an adjoint existence theorem of Butler, which can be found in \cite{barr2000toposes}---we invite the curious reader to view the full general statement \cite[Theorem 7.4.b]{barr2000toposes}. 

\begin{proposition}\label{weakadjprop} \cite[Theorem 2]{johnstone1975adjoint_lifting} Let $\lambda$ be a distributive law of a functor $\mathsf{R}:\mathbb{X} \rightarrow \mathbb{X}$ over a monad $(\mathsf{S}, \mu, \eta)$, and suppose that $\mathsf{R}$ has a left adjoint $\mathsf{L}$.  If $\mathbb{X}^\mathsf{S}$ admits reflexive coequalizers then the lifting of $\mathsf{R}$, $\overline{\mathsf{R}}: \mathbb{X}^\mathsf{S} \to \mathbb{X}^\mathsf{S}$, has a left adjoint $\mathsf{G}: \mathbb{X}^\mathsf{S} \to \mathbb{X}^\mathsf{S}$ such that $\mathsf{G}(\mathsf{S}(A), \mu_A) = (\mathsf{S}\mathsf{L}(A), \mu_{\mathsf{L}(A)})$. 
\end{proposition} 

Applying Proposition \ref{weakadjprop} to the Eilenberg-Moore category of a codifferential category, we obtain the following result: 

\begin{proposition} \label{propadj2} Let $\mathbb{X}$ be a codifferential category with algebra modality $(\mathsf{S}, \eta, \mu, \nabla, \mathsf{u})$ and deriving transformation $\mathsf{d}$, and suppose that $\mathbb{X}$ admits finite biproducts and $\mathbb{X}^\mathsf{S}$ admits reflexive coequalizers. Then for each $n \in \mathbb{N}$, $\overline{\mathsf{T}}_n: \mathbb{X}^\mathsf{S} \to \mathbb{X}^\mathsf{S}$ has a left adjoint. \end{proposition} 

Applying Theorem \ref{tanadj} to the above proposition, we obtain the main result of this paper: 

\begin{theorem}\label{thm:coem_dc_tan} If the coEilenberg-Moore category of a differential category with finite biproducts admits coreflexive equalizers (the dual of reflexive coequalizers), then the coEilenberg-Moore category is a tangent category. 
\end{theorem} 

For differential storage categories, in order to show that the coEilenberg-Moore category is a representable tangent category, we will need to look at the construction of Section \ref{EMCodiffsec} for codifferential categories with comonoidal algebra modalities. So let $\mathbb{X}$ be a codifferential category with a comonoidal algebra modality $(\mathsf{S}, \mu, \eta, \nabla, \mathsf{u}, \mathsf{n}, \mathsf{n}_K)$ such that $\mathbb{X}$ also admits finite biproducts, noting that $\mathsf{n}$ and $\mathsf{n}_K$ denote the comonoidal structure on $\mathsf{S}$ (Section \ref{coalgref}). Note that in $\mathbb{X}$, it follows from distributivity between the biproduct and monoidal product (which is automatic in any additive symmetric monoidal category) that for every $n \in \mathbb{N}$ one has that $\mathsf{T}_n(A) = \bigoplus\limits^n_{i=0} A \cong \bigoplus\limits^n_{i=0} (A \otimes K) \cong A \otimes \bigoplus\limits^n_{i=0} K = A \otimes \mathsf{T}_n(K)$. In particular when $n=1$, $\mathsf{T}(A) \cong A \otimes (K \oplus K)$. Recall that for a comonoidal algebra modality, $\otimes$ is a coproduct in the Eilenberg-Moore category and there is a map $\mathsf{n}_K: \mathsf{S}(K) \to K$ making $(K, \mathsf{n}_K)$ into an $\mathsf{S}$-algebra and an initial object. Then for every $n \in \mathbb{N}$ one has the following isomorphisms in the Eilenberg-Moore category $\mathbb{X}^\mathsf{S}$: $\overline{\mathsf{T}}_n(A, \nu) \cong (A, \nu) \otimes \overline{\mathsf{T}}_n(K, \mathsf{n}_K)$, and in particular, for $n=1$, $\overline{\mathsf{T}}(A, \nu) \cong (A, \nu) \otimes \overline{\mathsf{T}}(K, \mathsf{n}_K) = (A, \nu) \otimes (K \oplus K, \mathsf{n}^\flat_K)\;.$
 If $\mathbb{X}^\mathsf{S}$ admits reflexive coequalizers then by Proposition \ref{propadj2}, each functor $\overline{\mathsf{T}}_n \cong - \otimes \overline{\mathsf{T}}_n(K, \mathsf{n}_K)$ has a left adjoint. Dualizing, this implies that $- \otimes \overline{\mathsf{T}}_n(K, \mathsf{n}_K): (\mathbb{X}^\mathsf{S})^{op} \to (\mathbb{X}^\mathsf{S})^{op}$ has a right adjoint and therefore that $\overline{\mathsf{T}}_n(K, \mathsf{n}_K)$ is an exponent object in $(\mathbb{X}^\mathsf{S})^{op}$. Therefore, in view of Theorem \ref{thm:coem_dc_tan}, $(\mathbb{X}^\mathsf{S})^{op}$ is a representable tangent category whose infinitesimal object is $\overline{\mathsf{T}}(K, \mathsf{n}_K) = (K \oplus K,  \mathsf{n}^\flat_K)$. 
 
Let us restate this result in terms of differential storage categories. Let $\mathbb{X}$ be a differential storage category with coalgebra modality $(\oc, \delta, \varepsilon, \Delta, \mathsf{e})$ equipped with deriving transformation $\mathsf{d}_A: \oc(A) \otimes A \to \oc(A)$ and such that $\mathbb{X}$ has finite biproducts $\oplus$. Recall that the monoidal structure on $\oc$ includes a map $\mathsf{m}_K: K \to \oc(K)$ that makes $(K, \mathsf{m}_K)$ into a $\oc$-coalgebra. If $\mathbb{X}^\oc$ admits coreflexive equalizers, then $\mathbb{X}^\oc$ is a representable tangent category whose infinitesimal object is $(K \oplus K, \mathsf{m}^\sharp_K)$, where $\mathsf{m}^\sharp_K: K \oplus K \to \oc(K \oplus K)$ is defined as the unique map that makes the following diagram commute (using the couniversal property of the coproduct): 
$$ \xymatrixcolsep{1.5pc}\xymatrixrowsep{1pc}\xymatrix{ K\ar[dd]_-{\mathsf{m}_K} & K \oplus K \ar@{<-}[l]_-{\iota_0} \ar@{-->}[dd]^-{\mathsf{m}^\sharp_K} \ar@{<-}[r]^-{\iota_1} & K \ar[r]^-{\cong} & K \otimes K \ar[d]^-{\mathsf{m}_K \otimes 1_K}  \\
  &&& \oc(K) \otimes K \ar[d]^-{\oc(\iota_0) \otimes \iota_1} \\
 \oc(K) \ar[r]_-{\oc(\iota_0)} &\oc(K \oplus K) && \oc(K \oplus K) \otimes (K \oplus K) \ar[ll]^-{\mathsf{d}_{K \oplus K}} 
  } $$ 
 We summarize this result for differential storage categories to obtain the final main result of this paper: 

\begin{theorem}\label{reptanthm} If the coEilenberg-Moore category of a differential storage category admits coreflexive equalizers, then the coEilenberg-Moore category is a representable tangent category.  \end{theorem} 

\begin{example} \normalfont We conclude this section by looking briefly at some examples.
 \begin{enumerate}
\item For a field $k$ and Example \ref{codiffex}.\ref{MODdiffex}, recall once again that $\mathsf{VEC}^\mathsf{Sym}_k \cong \mathsf{CALG}_k$. It is well known that $\mathsf{CALG}_k$ is complete and cocomplete, and therefore $\mathsf{VEC}^\mathsf{Sym}_k$ admits reflexive coequalizers. Applying Theorem \ref{reptanthm} to this example, one obtains precisely the tangent structure on $\mathsf{CALG}_k^{op}$ from Example \ref{reptaneex}.\ref{Crigop} (and described in full detail in \cite[Proposition 5.16]{cockett2014differential}), where we recall that the infinitesimal object is the ring of dual numbers $k[\epsilon]$. It is interesting to note that for polynomial rings $k[X]$ we have in $\mathsf{CALG}_k^{op}$ that ${k[X]^{k[\epsilon]} \cong k[X] \otimes k[X]}$. In \cite{roryCT2014, roryFMCS}, the third author generalized the tangent structure on the opposite of the category of commutative $k$-algebras to the setting of certain codifferential categories, using universal derivations \cite{blute2015derivations}, which generalize K{\"a}hler differentials. 
\item Similarly for Example \ref{diffex}.\ref{cofreediffex} (again for a field $k$) recall that $\mathsf{VEC}^\oc_k$ is isomorphic to the category of cocommutative $k$-coalgebras, which is both complete and cocomplete \cite{porst2006corings,agore2011limits}. Therefore, $\mathsf{VEC}^\oc_k$ admits coreflexive equalizers and so by applying Theorem \ref{reptanthm},  $\mathsf{VEC}^\oc_k$ is a representable tangent category. Most interesting is that the infinitesimal object in this case is again the ring of dual numbers $k[\epsilon]$ but seen as a cocommutative $k$-coalgebra with comultiplication defined on the basis elements as $1 \mapsto 1 \otimes 1$ and $\epsilon \mapsto 1 \otimes \epsilon + \epsilon \otimes 1$. The coalgebra $k[\epsilon]$ played an important role in \cite{clift2017cofree}. 
\item For $C^\infty$-rings and Example \ref{codiffex}.\ref{Cinfdiffex}, the Eilenberg-Moore category is the category of $C^\infty$-rings.  As it is the category of algebras for a Lawvere theory, the category of $C^\infty$-rings is both complete and complete and therefore, in particular, admits reflexive coequalizers. Hence it follows that the opposite of the category of $C^\infty$-rings is a tangent category. In particular, for a smooth manifold $M$, the image of $C^\infty(M)$ under the tangent functor in this case is precisely $C^\infty(\mathsf{T}(M))$, where $\mathsf{T}(M)$ is the standard tangent bundle over $M$. 
\end{enumerate}
\end{example}

\section{Conclusion}\label{conclusion}

Given that differential categories involve a comonad it seems obvious from a categorical perspective that one should consider the coEilenberg-Moore category 
of coalgebras. That these coEilenberg-Moore categories are tangent categories, assuming the existence of coreflexive equalizers, provides an important way of generating tangent categories 
that is already ``baked in'' to algebraic geometry and synthetic differential geometry. As there are many examples of differential categories from various fields, this opens the door to studying new and interesting tangent categories. 

However, viewing these structures at this level of generality raises a number of further questions. Indeed, the fact that one has a tangent category begs the question of how various devices from abstract differential geometry (such as vector fields, Lie algebras, connections, solutions to differential equations, etc.) manifest in these settings. For example when and how do ``curve objects'' and ``line objects'' appear in these settings?  In any tangent category, one can also consider \emph{differential objects}, which are objects $A$ such that $\mathsf{T}(A) \cong A \times A$.  Can one characterize the differential objects in a coEilenberg-Moore category of a differential category? The cofree $\oc$-coalgebras $(\oc(A), \delta_A)$ are always differential objects, but when are these {\em exactly\/} the differential objects?  

There is now much work to be done to examine the more detailed ramifications of the constructions in this paper and to place specific results in a more general geography.

\subparagraph{\textbf{Acknowledgements:}} The authors would like thank to thank Steve Lack for pointing us to an adjoint lifting theorem of Butler found in Barr and Wells' book \cite{barr2000toposes}, as well as the anonymous referee of CSL2020 for pointing us to Johnstone's adjoint lifting theorem \cite{johnstone1975adjoint_lifting}.

\newpage
\bibliography{emdc}   % name your BibTeX data base
\bibliographystyle{spmpsci}      % mathematics and physical sciences

\appendix
\section{Proof of Proposition \ref{Codifflambda}}\label{Appendix}

For readability we omit the subscripts throughout the calculations of this section. We wish to show that $\lambda$ is a distributive law and satisfies the tangent monad axioms. To help simplify the calculations we denote $\mathsf{d}^\circ = (1 \otimes \eta)\nabla$, which sometimes known as the \textbf{coderiving transformation} \cite{cockett_lemay_2018}. This natural transformation satisfies the following equalities \cite{cockett_lemay_2018}: 
\[(\mathsf{u} \otimes 1)\mathsf{d}^\circ = \eta \quad \quad (\nabla \otimes 1)\mathsf{d}^\circ = (1 \otimes \mathsf{d}^\circ)\nabla \quad \quad \mathsf{d}^\circ\mathsf{d} = (1 \otimes 1) + (\mathsf{d} \otimes 1)(1 \otimes \tau)(\mathsf{d}^\circ \otimes 1)\]
Therefore $\lambda$ can be re-expressed as $\lambda := \mathsf{S}(\pi_0) \iota_0 + \mathsf{d}(\mathsf{S}(\pi_0) \otimes \pi_1)\mathsf{d}^\circ \iota_1$. We first show that $\lambda$ satisfies the distributive law identities. 

\noindent $\bullet$ $\mathsf{S}(\lambda)\lambda\mathsf{T}(\mu) = \mu \lambda$: Here we use the chain rule \textbf{[d.5]}. 
\begin{align*}
\mathsf{S}(\lambda)\lambda\mathsf{T}(\mu) &= \mathsf{S}(\lambda)\left(\mathsf{S}(\pi_0) \iota_0 + \mathsf{d}(\mathsf{S}(\pi_0) \otimes \pi_1)\mathsf{d}^\circ \iota_1 \right)(\mu \oplus \mu) \\
&= \mathsf{S}(\lambda)\mathsf{S}(\pi_0) \iota_0(\mu \oplus \mu) + \mathsf{S}(\lambda)\mathsf{d}(\mathsf{S}(\pi_0) \otimes \pi_1)\mathsf{d}^\circ \iota_1(\mu \oplus \mu) \\
&= \mathsf{S}(\lambda)\mathsf{S}(\pi_0) \iota_0(\mu \oplus \mu) + \mathsf{d}(\mathsf{S}(\lambda) \otimes \lambda)(\mathsf{S}(\pi_0) \otimes \pi_1)\mathsf{d}^\circ \iota_1(\mu \oplus \mu) \\
&= \mathsf{S}^2(\pi_0) \mu \iota_0+ \mathsf{d}\left(\mathsf{S}^2(\pi_0) \otimes \left( \mathsf{d}(\mathsf{S}(\pi_0) \otimes \pi_1)\mathsf{d}^\circ \right) \right)\mathsf{d}^\circ \mu \iota_1 \\
&= \mu \mathsf{S}(\pi_0) \iota_0+ \mathsf{d}(1 \otimes \mathsf{d})(\mathsf{S}^2(\pi_0) \otimes \mathsf{S}(\pi_0) \otimes \pi_1)(1 \otimes \mathsf{d}^\circ ) \mathsf{d}^\circ \mu \iota_1 \\
&= \mu \mathsf{S}(\pi_0) \iota_0+ \mathsf{d}(1 \otimes \mathsf{d})(\mathsf{S}^2(\pi_0) \otimes \mathsf{S}(\pi_0) \otimes \pi_1)(\mu \otimes \mathsf{d}^\circ )\nabla \iota_1 \\
&= \mu \mathsf{S}(\pi_0) \iota_0+ \mathsf{d}(\mu \otimes \mathsf{d})(\mathsf{S}(\pi_0) \otimes \mathsf{S}(\pi_0) \otimes \pi_1)(1 \otimes \mathsf{d}^\circ )\nabla \iota_1 \\
&= \mu \mathsf{S}(\pi_0) \iota_0+ \mathsf{d}(\mu \otimes \mathsf{d})(\mathsf{S}(\pi_0) \otimes \mathsf{S}(\pi_0) \otimes \pi_1)(\nabla \otimes 1 )\mathsf{d}^\circ \iota_1 \\
&= \mu \mathsf{S}(\pi_0) \iota_0+ \mathsf{d}(\mu \otimes \mathsf{d}))(\nabla \otimes 1 )(\mathsf{S}(\pi_0) \otimes \pi_1)\mathsf{d}^\circ \iota_1 \\
&= \mu \mathsf{S}(\pi_0) \iota_0+ \mu \mathsf{d}(\mathsf{S}(\pi_0) \otimes \pi_1)\mathsf{d}^\circ \iota_1 \\
&= \mu\left(\mathsf{S}(\pi_0) \iota_0 + \mathsf{d}(\mathsf{S}(\pi_0) \otimes \pi_1)\mathsf{d}^\circ \iota_1 \right)\\
&= \mu \lambda
\end{align*}
$\bullet$ $\eta \lambda = \mathsf{T}(\eta)$: Here we use the linear rule  \textbf{[d.3]}.
\begin{align*}
\eta \lambda&=~\eta\left(\mathsf{S}(\pi_0) \iota_0 + \mathsf{d}(\mathsf{S}(\pi_0) \otimes \pi_1)\mathsf{d}^\circ \iota_1 \right)\\
&=~\eta \mathsf{S}(\pi_0) \iota_0+ \eta \mathsf{d}(\mathsf{S}(\pi_0) \otimes \pi_1)\mathsf{d}^\circ \iota_1 \\
&=~\pi_0 \eta \iota_0+ (\mathsf{u} \otimes 1)(\mathsf{S}(\pi_0) \otimes \pi_1)\mathsf{d}^\circ \iota_1 \\
&=~\pi_0 \eta \iota_0+ \pi_1 (\mathsf{u} \otimes 1) \mathsf{d}^\circ \iota_1 \\
&=~\pi_0 \eta \iota_0+ \pi_1 \eta \iota_1 \\
&=~ \eta \oplus \eta\\
&=~ \mathsf{T}(\eta)
\end{align*}
Next we show the tangent monad identities: 

\noindent $\bullet$ $\lambda p = \mathsf{S}(p)$: Here we use the biproduct identities. 
\begin{align*}
\lambda \mathsf{p} &=~ (\mathsf{S}(\pi_0) \iota_0 + \mathsf{d}(\mathsf{S}(\pi_0) \otimes \pi_1)\mathsf{d}^\circ \iota_1) \pi_0\\
&=~ \mathsf{S}(\pi_0) \iota_0\pi_0 + \mathsf{d}(\mathsf{S}(\pi_0) \otimes \pi_1)\mathsf{d}^\circ \iota_1\pi_0 \\
&=~ \mathsf{S}(\pi_0) \iota_0\pi_0 + \mathsf{d}(\mathsf{S}(\pi_0) \otimes \pi_1)\mathsf{d}^\circ 0 \\
&=~ \mathsf{S}(\pi_0) + 0 \\
&=~ \mathsf{S}(\pi_0) \\
&=~ \mathsf{S}(\mathsf{p})
\end{align*}
\noindent $\bullet$ $\mathsf{S}(\sigma) \lambda  = \langle \mathsf{S}(\rho_0)\lambda, \mathsf{S}(\rho_1)\lambda \rangle \sigma$: The pullback projections $\rho_0$ and $\rho_1$ are respectively defined as \\ \noindent $\xymatrixcolsep{2pc}\xymatrix{A \oplus A \oplus A \ar[r]^-{1 \oplus \pi_i} & A \oplus A}$ 
with $i \in \lbrace 0,1 \rbrace$. One can check that we have the following equality (where recall that here $\langle -, - \rangle$ is the pullback pairing) 
$$\langle \mathsf{S}(\rho_0) \lambda, \mathsf{S}(\rho_1) \lambda \rangle = \mathsf{S}(\pi_0) \iota_0 + \mathsf{d}(\mathsf{S}(\pi_0) \otimes \pi_1) \mathsf{d}^\circ \iota_1 + \mathsf{d}(\mathsf{S}(\pi_0) \otimes \pi_2) \mathsf{d}^\circ \iota_2 $$ 
and note that by definition: 
$$\sigma \iota_0 = \iota_0 \quad \quad \sigma \iota_1 = \iota_1 \quad \quad \sigma \iota_2 = \iota_1 \quad \quad \sigma \pi_0= \pi_0 \quad \quad \sigma \pi_1 = \pi_1 + \pi_2 $$ 
Then we have that: 
\begin{align*}
\mathsf{S}(\sigma) \lambda &= \mathsf{S}(\sigma) \left(\mathsf{S}(\pi_0) \iota_0 + \mathsf{d}(\mathsf{S}(\pi_0) \otimes \pi_1)\mathsf{d}^\circ\iota_1 \right) \\
&= \mathsf{S}(\sigma)\mathsf{S}(\pi_0) \iota_0 + (\sigma) \mathsf{d}(\mathsf{S}(\pi_0) \otimes \pi_1)\mathsf{d}^\circ \iota_1  \\ 
&= \mathsf{S}(\pi_0) \iota_0 +   \mathsf{d}\left( \mathsf{S}(\sigma) \otimes \sigma \right)(\mathsf{S}(\pi_0) \otimes \pi_1)\mathsf{d}^\circ \iota_1  \\
&= \mathsf{S}(\pi_0) \iota_0 +   \mathsf{d}\left(\mathsf{S}(\pi_0) \otimes (\pi_1 + \pi_2) \right)\mathsf{d}^\circ\iota_1  \\
&= \mathsf{S}(\pi_0) \iota_0 +   \mathsf{d}(\mathsf{S}(\pi_0) \otimes \pi_1) \mathsf{d}^\circ \iota_1 +  \mathsf{d}(\mathsf{S}(\pi_0) \otimes \pi_2)\mathsf{d}^\circ \iota_1   \\
&= \mathsf{S}(\pi_0) \iota_0\sigma+   \mathsf{d}(\mathsf{S}(\pi_0) \otimes \pi_1) \mathsf{d}^\circ \iota_1\sigma +  \mathsf{d}(\mathsf{S}(\pi_0) \otimes \pi_2) \mathsf{d}^\circ \iota_2 \sigma   \\
&= (\mathsf{S}(\pi_0) \iota_0 + \mathsf{d}(\mathsf{S}(\pi_0) \otimes \pi_1) \mathsf{d}^\circ \iota_1 + \mathsf{d}(\mathsf{S}(\pi_0) \otimes \pi_2) \mathsf{d}^\circ \iota_2) \sigma\\
&=  \langle \mathsf{S}(\rho_0)\lambda, \mathsf{S}(\rho_1)\lambda \rangle \sigma
\end{align*}
$\bullet$ $\mathsf{S}(z) \lambda = z$: Here we use the biproduct identities: 
\begin{align*}
\mathsf{S}(z) \lambda &= \mathsf{S}(\iota_0) \left(\mathsf{S}(\pi_0) \iota_0 + \mathsf{d}(\mathsf{S}(\pi_0) \otimes \pi_1) \mathsf{d}^\circ \iota_1 \right) \\
&=  \mathsf{S}(\iota_0)\mathsf{S}(\pi_0) \iota_0 +  \mathsf{S}(\iota_0) \mathsf{d}(\mathsf{S}(\pi_0) \otimes \pi_1) \mathsf{d}^\circ \iota_1 \\
&= \iota_0 + \mathsf{d}(\mathsf{S}(\iota_0) \otimes \iota_0)(\mathsf{S}(\pi_0) \otimes \pi_1) \mathsf{d}^\circ\iota_1\\ 
&= \iota_0 + \mathsf{d}(1 \otimes 0) \mathsf{d}^\circ \iota_1\\ 
&= \iota_0 + 0 \\
&= \iota_0 \\
&= z
\end{align*}
For the remaining two tangent monad identities, it will be useful to simplify $\lambda \mathsf{T}(\lambda)$ first:
\begin{align*}
& \lambda \mathsf{T}(\lambda)=~ \left(\mathsf{S}(\pi_0) \iota_0 + \mathsf{d}(\mathsf{S}(\pi_0) \otimes \pi_1)\mathsf{d}^\circ \iota_1 \right)(\lambda \oplus \lambda)  (\lambda \oplus \lambda) \\
&=~  \mathsf{S}(\pi_0) \iota_0(\lambda \oplus \lambda) + \mathsf{d}(\mathsf{S}(\pi_0) \otimes \pi_1)\mathsf{d}^\circ \iota_1(\lambda \oplus \lambda)\\
&=~ \mathsf{S}(\pi_0) \lambda \iota_0 + \mathsf{d}(\mathsf{S}(\pi_0) \otimes \pi_1)\mathsf{d}^\circ \lambda \iota_1  \\
&=~ \mathsf{S}(\pi_0) \left(\mathsf{S}(\pi_0) \iota_0 + \mathsf{d}(\mathsf{S}(\pi_0) \otimes \pi_1)\mathsf{d}^\circ \iota_1 \right) \iota_0 \\
&+ \mathsf{d}(\mathsf{S}(\pi_0) \otimes \pi_1)\mathsf{d}^\circ \left(\mathsf{S}(\pi_0) \iota_0 + \mathsf{d}(\mathsf{S}(\pi_0) \otimes \pi_1)\mathsf{d}^\circ \iota_1 \right) \iota_1  \\
&=~ \mathsf{S}(\pi_0) \mathsf{S}(\pi_0) \iota_0 \iota_0 + \mathsf{S}(\pi_0) \mathsf{d}(\mathsf{S}(\pi_0) \otimes \pi_1)\mathsf{d}^\circ \iota_1 \iota_0   \\&+~ \mathsf{d}(\mathsf{S}(\pi_0) \otimes \pi_1)\mathsf{d}^\circ \mathsf{S}(\pi_0) \iota_0 \iota_1 + \mathsf{d}(\mathsf{S}(\pi_0) \otimes \pi_1)\mathsf{d}^\circ\mathsf{d}(\mathsf{S}(\pi_0) \otimes \pi_1)\mathsf{d}^\circ \iota_1\iota_1  \\
&=~ \mathsf{S}(\pi_0) \mathsf{S}(\pi_0) \iota_0 \iota_0 +  \mathsf{d}(\mathsf{S}(\pi_0) \otimes \pi_0)(\mathsf{S}(\pi_0) \otimes \pi_1)\mathsf{d}^\circ \iota_0 \iota_1  \\&
+ \mathsf{d}(\mathsf{S}(\pi_0) \otimes \pi_1)(\mathsf{S}(\pi_0) \otimes \pi_0)\mathsf{d}^\circ  \iota_1 \iota_0 + \mathsf{d}(\mathsf{S}(\pi_0) \otimes \pi_1)(\mathsf{S}(\pi_0) \otimes \pi_1)\mathsf{d}^\circ \iota_1\iota_1  \\
&+  \mathsf{d}(\mathsf{S}(\pi_0) \otimes \pi_1)(\mathsf{d} \otimes 1)(1 \otimes \tau)(\mathsf{d}^\circ \otimes 1)(\mathsf{S}(\pi_0) \otimes \pi_1)\mathsf{d}^\circ \iota_1\iota_1\\
&=~ \mathsf{S}(\pi_0) \mathsf{S}(\pi_0) \iota_0 \iota_0 +  \mathsf{d}(\mathsf{S}(\pi_0) \otimes \pi_0)(\mathsf{S}(\pi_0) \otimes \pi_1)\mathsf{d}^\circ \iota_0 \iota_1  \\
& + \mathsf{d}(\mathsf{S}(\pi_0) \otimes \pi_1)(\mathsf{S}(\pi_0) \otimes \pi_0)\mathsf{d}^\circ  \iota_1 \iota_0 + \mathsf{d}(\mathsf{S}(\pi_0) \otimes \pi_1)(\mathsf{S}(\pi_0) \otimes \pi_1)\mathsf{d}^\circ \iota_1\iota_1  \\
&+  \mathsf{d}(\mathsf{d} \otimes 1)(\mathsf{S}(\pi_0) \otimes \pi_1 \otimes \pi_0)(\mathsf{S}(\pi_0) \otimes \pi_0 \otimes \pi_1)(\mathsf{d}^\circ \otimes 1)\mathsf{d}^\circ \iota_1\iota_1
\end{align*}
$\bullet$ $\mathsf{S}(\ell)\lambda \mathsf{T}(\lambda)=\lambda \ell$: 
\begin{align*}
&\mathsf{S}(\ell)\lambda \mathsf{T}(\lambda) =~ \mathsf{S}(\iota_0 \oplus \iota_1) \bigg( \mathsf{S}(\pi_0) \mathsf{S}(\pi_0) \iota_0 \iota_0 +  \mathsf{d}(\mathsf{S}(\pi_0) \otimes \pi_0)(\mathsf{S}(\pi_0) \otimes \pi_1)\mathsf{d}^\circ \iota_0 \iota_1 \\
& + \mathsf{d}(\mathsf{S}(\pi_0) \otimes \pi_1)(\mathsf{S}(\pi_0) \otimes \pi_0)\mathsf{d}^\circ  \iota_1 \iota_0 + \mathsf{d}(\mathsf{S}(\pi_0) \otimes \pi_1)(\mathsf{S}(\pi_0) \otimes \pi_1)\mathsf{d}^\circ \iota_1\iota_1  \\
&+  \mathsf{d}(\mathsf{d} \otimes 1)(\mathsf{S}(\pi_0) \otimes \pi_1 \otimes \pi_0)(\mathsf{S}(\pi_0) \otimes \pi_0 \otimes \pi_1)(\mathsf{d}^\circ \otimes 1)\mathsf{d}^\circ \iota_1\iota_1 \bigg) \\
&=~ \mathsf{S}(\iota_0 \oplus \iota_1)\mathsf{S}(\pi_0) \mathsf{S}(\pi_0) \iota_0 \iota_0 +  \mathsf{S}(\iota_0 \oplus \iota_1)\mathsf{d}(\mathsf{S}(\pi_0) \otimes \pi_0)(\mathsf{S}(\pi_0) \otimes \pi_1)\mathsf{d}^\circ \iota_0 \iota_1  \\
& + \mathsf{S}(\iota_0 \oplus \iota_1)\mathsf{d}(\mathsf{S}(\pi_0) \otimes \pi_1)(\mathsf{S}(\pi_0) \otimes \pi_0)\mathsf{d}^\circ  \iota_1 \iota_0 \\
&+ \mathsf{S}(\iota_0 \oplus \iota_1)\mathsf{d}(\mathsf{S}(\pi_0) \otimes \pi_1)(\mathsf{S}(\pi_0) \otimes \pi_1)\mathsf{d}^\circ \iota_1\iota_1  \\
&+  \mathsf{S}(\iota_0 \oplus \iota_1)\mathsf{d}(\mathsf{d} \otimes 1)(\mathsf{S}(\pi_0) \otimes \pi_1 \otimes \pi_0)(\mathsf{S}(\pi_0) \otimes \pi_0 \otimes \pi_1)(\mathsf{d}^\circ \otimes 1)\mathsf{d}^\circ \iota_1\iota_1 \\
&=~ \mathsf{S}(\pi_0)\mathsf{S}(\iota_0) \mathsf{S}(\pi_0) \iota_0 \iota_0 +  \mathsf{d}\left( \mathsf{S}(\iota_0 \oplus \iota_1) \otimes (\iota_0 \oplus \iota_1) \right)(\mathsf{S}(\pi_0) \otimes \pi_0)(\mathsf{S}(\pi_0) \otimes \pi_1)\mathsf{d}^\circ \iota_0 \iota_1  \\
&+ \mathsf{d}\left( \mathsf{S}(\iota_0 \oplus \iota_1) \otimes (\iota_0 \oplus \iota_1) \right)(\mathsf{S}(\pi_0) \otimes \pi_1)(\mathsf{S}(\pi_0) \otimes \pi_0)\mathsf{d}^\circ  \iota_1 \iota_0 \\
&+ \mathsf{d}\left( \mathsf{S}(\iota_0 \oplus \iota_1) \otimes (\iota_0 \oplus \iota_1) \right)(\mathsf{S}(\pi_0) \otimes \pi_1)(\mathsf{S}(\pi_0) \otimes \pi_1)\mathsf{d}^\circ \iota_1\iota_1  \\
&+ \mathsf{d}(\mathsf{d} \otimes 1)\left( \mathsf{S}(\iota_0 \oplus \iota_1) \otimes (\iota_0 \oplus \iota_1) \otimes (\iota_0 \oplus \iota_1) \right)(\mathsf{S}(\pi_0) \otimes \pi_1 \otimes \pi_0)
(\mathsf{S}(\pi_0) \otimes \pi_0 \otimes \pi_1)\\
&~~~(\mathsf{d}^\circ \otimes 1)\mathsf{d}^\circ \iota_1\iota_1 \\
&=~ \mathsf{S}(\pi_0)\iota_0 \iota_0 +  \mathsf{d}(\mathsf{S}(\pi_0) \otimes \pi_0)(\mathsf{S}(\iota_0) \otimes \iota_0)(\mathsf{S}(\pi_0) \otimes \pi_1)\mathsf{d}^\circ \iota_0 \iota_1  \\
& + \mathsf{d}(\mathsf{S}(\pi_0) \otimes \pi_1)(\mathsf{S}(\iota_0) \otimes \iota_0)(\mathsf{S}(\pi_0) \otimes \pi_0)\mathsf{d}^\circ  \iota_1 \iota_0 \\
&+ (\mathsf{S}(\pi_0) \otimes \pi_1)(\mathsf{S}(\iota_0) \otimes \iota_1)(\mathsf{S}(\pi_0) \otimes \pi_1)\mathsf{d}^\circ \iota_1\iota_1  \\
&+ \mathsf{d}(\mathsf{d} \otimes 1)(\mathsf{S}(\pi_0) \otimes \pi_1 \otimes \pi_0)(\mathsf{S}(\iota_0) \otimes \iota_1 \otimes \iota_0)(\mathsf{S}(\pi_0) \otimes \pi_0 \otimes \pi_1)(\mathsf{d}^\circ \otimes 1)\mathsf{d}^\circ \iota_1\iota_1 \\
&=~ \mathsf{S}(\pi_0)\iota_0 \iota_0 +  \mathsf{d}(\mathsf{S}(\pi_0) \otimes \pi_0)(1 \otimes 0)\mathsf{d}^\circ \iota_0 \iota_1  \\
&+ \mathsf{d}(\mathsf{S}(\pi_0) \otimes \pi_1)\mathsf{d}^\circ  \iota_1 \iota_0 + (\mathsf{S}(\pi_0) \otimes \pi_1)(1 \otimes 0)\mathsf{d}^\circ \iota_1\iota_1  \\
&+ \mathsf{d}(\mathsf{d} \otimes 1)(\mathsf{S}(\pi_0) \otimes \pi_1 \otimes \pi_0)(1 \otimes 0 \otimes 0)(\mathsf{d}^\circ \otimes 1)\mathsf{d}^\circ \iota_1\iota_1 \\
&=~ \mathsf{S}(\pi_0)\iota_0 \iota_0 + 0 + \mathsf{d}(\mathsf{S}(\pi_0) \otimes \pi_1)\mathsf{d}^\circ  \iota_1 \iota_0 + 0 + 0 \\
&=~\mathsf{S}(\pi_0)\iota_0 \iota_0 + \mathsf{d}(\mathsf{S}(\pi_0) \otimes \pi_1)\mathsf{d}^\circ \iota_1\iota_1 \\
&=~ \mathsf{S}(\pi_0) \iota_0(\iota_0 \oplus \iota_1) + \mathsf{d}(\mathsf{S}(\pi_0) \otimes \pi_1)\mathsf{d}^\circ \iota_1(\iota_0 \oplus \iota_1)\\
&=~ \left(\mathsf{S}(\pi_0) \iota_0 + \mathsf{d}(\mathsf{S}(\pi_0) \otimes \pi_1)\mathsf{d}^\circ \iota_1 \right) (\iota_0 \oplus \iota_1) \\
&=~ \lambda \ell
\end{align*}
$\bullet$ $\lambda \mathsf{T}(\lambda)c= \mathsf{S}(c)\lambda \mathsf{T}(\lambda)$: First notice the following equalities:  
$$\iota_j \iota_k c = \iota_k \iota_j \quad \quad c\pi_j \pi_k = \pi_k \pi_j$$
Then we have that: 
\begin{align*}
&\lambda \mathsf{T}(\lambda)c =~  \mathsf{S}(\pi_0) \mathsf{S}(\pi_0) \iota_0 \iota_0c +  \mathsf{d}(\mathsf{S}(\pi_0) \otimes \pi_0)(\mathsf{S}(\pi_0) \otimes \pi_1)\mathsf{d}^\circ \iota_0 \iota_1c  \\
& + \mathsf{d}(\mathsf{S}(\pi_0) \otimes \pi_1)(\mathsf{S}(\pi_0) \otimes \pi_0)\mathsf{d}^\circ  \iota_1 \iota_0c + \mathsf{d}(\mathsf{S}(\pi_0) \otimes \pi_1)(\mathsf{S}(\pi_0) \otimes \pi_1)\mathsf{d}^\circ \iota_1\iota_1 c \\
&+  \mathsf{d}(\mathsf{d} \otimes 1)(\mathsf{S}(\pi_0) \otimes \pi_1 \otimes \pi_0)(\mathsf{S}(\pi_0) \otimes \pi_0 \otimes \pi_1)(\mathsf{d}^\circ \otimes 1)\mathsf{d}^\circ \iota_1\iota_1c \\
&=~ \mathsf{S}(\pi_0) \mathsf{S}(\pi_0) \iota_0 \iota_0+  \mathsf{d}(\mathsf{S}(\pi_0) \otimes \pi_0)(\mathsf{S}(\pi_0) \otimes \pi_1)\mathsf{d}^\circ \iota_1 \iota_0  \\
& + \mathsf{d}(\mathsf{S}(\pi_0) \otimes \pi_1)(\mathsf{S}(\pi_0) \otimes \pi_0)\mathsf{d}^\circ  \iota_0 \iota_1 + \mathsf{d}(\mathsf{S}(\pi_0) \otimes \pi_1)(\mathsf{S}(\pi_0) \otimes \pi_1)\mathsf{d}^\circ \iota_1\iota_1  \\
&+  \mathsf{d}(\mathsf{d} \otimes 1)(\mathsf{S}(\pi_0) \otimes \pi_1 \otimes \pi_0)(\mathsf{S}(\pi_0) \otimes \pi_0 \otimes \pi_1)(\mathsf{d}^\circ \otimes 1)\mathsf{d}^\circ \iota_1\iota_1 \\
&=~ \mathsf{S}(c)\mathsf{S}(\pi_0) \mathsf{S}(\pi_0) \iota_0 \iota_0 +  \mathsf{d}(\mathsf{S}(c) \otimes c)(\mathsf{S}(\pi_0) \otimes \pi_1)(\mathsf{S}(\pi_0) \otimes \pi_0)\mathsf{d}^\circ \iota_1 \iota_0  \\
& + \mathsf{d}(\mathsf{S}(c) \otimes c)(\mathsf{S}(\pi_0) \otimes \pi_0)(\mathsf{S}(\pi_0) \otimes \pi_1)\mathsf{d}^\circ  \iota_0 \iota_1 \\
&+ \mathsf{d}(\mathsf{S}(c) \otimes c)(\mathsf{S}(\pi_0) \otimes \pi_1)(\mathsf{S}(\pi_0) \otimes \pi_1)\mathsf{d}^\circ \iota_1\iota_1  \\
&+  \mathsf{d}(\mathsf{d} \otimes 1)(\mathsf{S}(c) \otimes c \otimes c)(\mathsf{S}(\pi_0) \otimes \pi_0 \otimes \pi_1)(\mathsf{S}(\pi_0) \otimes \pi_1 \otimes \pi_0)(\mathsf{d}^\circ \otimes 1)\mathsf{d}^\circ \iota_1\iota_1 \\
&=~ \mathsf{S}(c)\mathsf{S}(\pi_0) \mathsf{S}(\pi_0) \iota_0 \iota_0+  \mathsf{S}(c)\mathsf{d}(\mathsf{S}(\pi_0) \otimes \pi_1)(\mathsf{S}(\pi_0) \otimes \pi_0)\mathsf{d}^\circ \iota_1 \iota_0  \\
& + \mathsf{S}(c)\mathsf{d}(\mathsf{S}(\pi_0) \otimes \pi_0)(\mathsf{S}(\pi_0) \otimes \pi_1)\mathsf{d}^\circ  \iota_0 \iota_1 + \mathsf{S}(c)\mathsf{d}(\mathsf{S}(\pi_0) \otimes \pi_1)(\mathsf{S}(\pi_0) \otimes \pi_1)\mathsf{d}^\circ \iota_1\iota_1  \\
&+  \mathsf{S}(c)\mathsf{d}(\mathsf{d} \otimes 1)(\mathsf{S}(\pi_0) \otimes \pi_0 \otimes \pi_1)(\mathsf{S}(\pi_0) \otimes \pi_1 \otimes \pi_0)(\mathsf{d}^\circ \otimes 1)\mathsf{d}^\circ \iota_1\iota_1 \\
&=~ \mathsf{S}(c)\mathsf{S}(\pi_0) \mathsf{S}(\pi_0) \iota_0 \iota_0 +  \mathsf{S}(c)\mathsf{d}(\mathsf{S}(\pi_0) \otimes \pi_1)(\mathsf{S}(\pi_0) \otimes \pi_0)\mathsf{d}^\circ \iota_1 \iota_0  \\
& + \mathsf{S}(c)\mathsf{d}(\mathsf{S}(\pi_0) \otimes \pi_0)(\mathsf{S}(\pi_0) \otimes \pi_1)\mathsf{d}^\circ  \iota_0 \iota_1 + \mathsf{S}(c)\mathsf{d}(\mathsf{S}(\pi_0) \otimes \pi_1)(\mathsf{S}(\pi_0) \otimes \pi_1)\mathsf{d}^\circ \iota_1\iota_1  \\
&+  \mathsf{S}(c)\mathsf{d}(\mathsf{d} \otimes 1)(\mathsf{S}(\pi_0) \otimes \pi_1 \otimes \pi_0)(\mathsf{S}(\pi_0) \otimes \pi_0 \otimes \pi_1)(\mathsf{d}^\circ \otimes 1)\mathsf{d}^\circ \iota_1\iota_1 \\
&=~ \mathsf{S}(c) \bigg( \mathsf{S}(\pi_0) \mathsf{S}(\pi_0) \iota_0 \iota_0 +  \mathsf{d}(\mathsf{S}(\pi_0) \otimes \pi_0)(\mathsf{S}(\pi_0) \otimes \pi_1)\mathsf{d}^\circ \iota_0 \iota_1  \\
& + \mathsf{d}(\mathsf{S}(\pi_0) \otimes \pi_1)(\mathsf{S}(\pi_0) \otimes \pi_0)\mathsf{d}^\circ  \iota_1 \iota_0 + \mathsf{d}(\mathsf{S}(\pi_0) \otimes \pi_1)(\mathsf{S}(\pi_0) \otimes \pi_1)\mathsf{d}^\circ \iota_1\iota_1  \\
&+  \mathsf{d}(\mathsf{d} \otimes 1)(\mathsf{S}(\pi_0) \otimes \pi_1 \otimes \pi_0)(\mathsf{S}(\pi_0) \otimes \pi_0 \otimes \pi_1)(\mathsf{d}^\circ \otimes 1)\mathsf{d}^\circ \iota_1\iota_1 \bigg) \\
&=~ \mathsf{S}(c)\lambda \mathsf{T}(\lambda) 
\end{align*}
\end{document}